\documentclass{siamltex}

 \usepackage{amsmath, amssymb, graphicx}

 \usepackage{color}

\newcommand\R{\mathbb R} \newcommand\C{\mathbb C}
\newcommand\PP{\mathbb P}

\newcommand{\Diag}{\operatorname{Diag}}
\newcommand{\adj}{\operatorname{adj}}

\newcommand{\tr}{\operatorname{tr}}

\newcommand{\td}[1]{}

 \title{Tensor Rank, Invariants, Inequalities, and Applications}
 
\author{Elizabeth S.~Allman\footnotemark[2]
\and Peter D. Jarvis\footnotemark[3] 
\and John A.~Rhodes\footnotemark[2]
\and Jeremy G. Sumner\footnotemark[3] 
}
\begin{document}

\date\today

\maketitle

\renewcommand{\thefootnote}{\fnsymbol{footnote}}
\footnotetext[2]{Department of Mathematics and Statistics,
University of Alaska, PO Box 756660, Fairbanks AK 99775, USA ({\tt e.allman@alaska.edu, j.rhodes@alaska.edu}). Supported in part by US National Science Foundation grant DMS 0714830.}
\footnotetext[3]{School of Mathematics and Physics,
University of Tasmania, Private Bag 37, Hobart, TAS 7001, Australia ({\tt peter.jarvis@utas.edu.au, jsumner@utas.edu.au}). 
Supported in part by Australian Research Council grants DP0877447 and FT100100031.
}

\renewcommand{\thefootnote}{\aribic{footnote}}

\begin{abstract}
Though algebraic geometry over $\mathbb C$ is often used to describe the closure of the tensors of a given size and complex rank, this variety includes tensors of both smaller and larger rank. Here we focus on the $n\times n\times n$ tensors of rank $n$ over $\mathbb C$,  which has as a dense subset the orbit of a single tensor under a natural group action. We construct polynomial invariants under this group action whose non-vanishing distinguishes this orbit from points only in its closure. Together with an explicit subset of the defining polynomials of the variety, this gives a semialgebraic description of the tensors of rank $n$ and multilinear rank $(n,n,n)$.
The polynomials we construct coincide with  Cayley's hyperdeterminant in the case $n=2$, and thus generalize it. Though our construction is direct and explicit, we also recast our functions in the language of representation theory for additional insights.

We give three applications in different directions: First, we develop basic  topological understanding of how the real tensors of complex rank $n$ and multilinear rank $(n,n,n)$ form a collection of path-connected subsets, one of which contains tensors of real rank $n$. Second, we use the invariants to develop a semialgebraic description of the set of probability distributions that can arise from a simple stochastic model with a hidden variable, a model that is important in phylogenetics and other fields. Third, we construct simple examples of tensors of rank $2n-1$ which lie in the closure of those of rank $n$.
\end{abstract}



\pagestyle{myheadings}
\thispagestyle{plain}
\markboth{E.~S.~ALLMAN, P.~D.~JARVIS, J.~A.~RHODES, AND J.~G.~SUMNER}{TENSOR RANK, INVARIANTS, INEQUALITIES, AND APPLICATIONS }

\section{Introduction}

The notion of tensor rank naturally extends the familiar notion of matrix rank for two-dimensional numerical arrays  to $d$-dimensional arrays, and likewise has extensive connections to applied problems. However, basic questions about tensor rank can be much more difficult to answer than their matrix analogs, and many open problems remain.
Several natural problems concerning tensor rank are to  determine for a given field the rank of an explicitly given tensor, to determine for a given field and format the possible ranks of all tensors, and to determine for a given field and format the generic rank(s) of a tensor. While the matrix versions of these problems are solved by an understanding of Gaussian elimination and determinants, for higher dimensional tensors they have so far eluded general solutions.

\smallskip

The case of $2\times 2\times 2$ tensors, however, is quite well studied \cite{dSL}, and provides one model of desirable understanding: Over $\mathbb C$ or $\mathbb R$, such a tensor may have rank $0$, $1$, $2$, or $3$ only. Over $\C$ descriptions of the sets of tensors of each of these possible ranks may be given, in terms of intersections, unions, and complements of explicit algebraic varieties. These descriptions can be thus phrased as boolean combinations of polynomial equalities. Over $\R$, analogous explicit descriptions require polynomial inequalities using `$>$' as well, and the descriptions are thus semialgebraic. For larger tensors  of any given rank the existence of a semialgebraic description is a consequence of the Tarski-Seidenberg Theorem  \cite{Tarski1951, Seidenberg1954}, but complete explicit descriptions are not known.

A polynomial
of particular importance in understanding the  $2\times 2\times 2$ case is \emph{Cayley's hyperdeterminant} \cite{Cayley}, 
\begin{multline}
\Delta(P)=(p_{000}^2p_{111}^2+p_{001}^2p_{110}^2+p_{010}^2p_{101}^2+p_{011}^2p_{100}^2)
\\ -2(p_{000}p_{001}p_{110}p_{111}+
p_{000}p_{010}p_{101}p_{111}+p_{000}p_{011}p_{100}p_{111} \\
+p_{001}p_{010}p_{101}p_{110}+p_{001}p_{011}p_{110}p_{100}+p_{010}p_{011}p_{101}p_{100})\\
+4(p_{000}p_{011}p_{101}p_{110}+p_{001}p_{010}p_{100}p_{111}),\label{eq:hyper}
\end{multline}
also known as the \emph{tangle} in the physics literature \cite{coffman2000}. The function $\Delta$ has  non-zero values precisely on a certain dense subset of those $2\times 2\times 2$ tensors of complex tensor rank 2. If a tensor is real, the sign of the $\Delta$ further indicates information about its tensor rank  over $\mathbb R$: If $\Delta>0$, then the tensor has real tensor rank $2$, and if $\Delta<0$ its real tensor rank is 3. 

The role of $\Delta$ here can be partially understood as a consequence of it being an invariant of the group $GL(2,\C)\times GL(2,\C)\times GL(2,\C)$, which acts on $2\times 2\times 2$ complex tensors in the three indices, and preserves their tensor rank. The transformation property of $\Delta$ under this group, along with explicit evaluation at a particular tensor of rank 2,  implies it is non-zero on an orbit which is dense among all tensors of rank 2. The fact that it is zero off of this orbit can be shown by first determining a list of canonical representatives of other orbits, and then explicitly evaluating $\Delta$ on them to see that it vanishes. Thus both the transformation property of $\Delta$ under the group and the  ability to evaluate $\Delta$ at specific points are essential.

\medskip

In this work we focus on $n\times n\times n$ tensors of tensor rank $n$, over $\mathbb C$ and over $\mathbb R$, with the goal of
generalizing our detailed understanding of the $2\times 2\times 2$ tensors to this particular case. Although we do not translate our results here to the cases of 
$n_1\times n_2\times n_3$ tensors of rank $n$ with $n\le n_i$, this should be possible by applying maps of $\C^{n_i}\to\C^n$. Thus what is most important about this case is that
the dimensions of the tensor are sufficiently large that they do not put constrictions on studying the given rank $n$. (A more careful reading will show that for many arguments we only  need $n_i\ge n$ for at least $2$ values of $i$.)

Over $\C$, the rank-$2$ tensors are dense among all $2\times 2\times 2$ tensors. However, for $n>2$, the closure of the rank-$n$ tensors
in the $n\times n\times n$ ones forms an algebraic variety of dimension strictly less than $n^3$. From the closure operation, this variety contains all tensors of rank $<n$, as well as some of rank $>n$. 
Much previous work has focused on determining defining polynomials of this variety, that is, polynomials that vanish on all such rank $n$ tensors. For $n=3$, the ideal of polynomials defining this variety is known \cite{GSS}. For $n=4$ a set-theoretic defining set of polynomials has been determined \cite{Friedland,BatesOeding,FriedlandGross}. For all  $n\ge3$, many polynomials in the ideal are known, through a general construction of `commutation relations'  \cite{AR03,ARgm}. Moreover, the commutation relations give the full ideal up to an explicit saturation, and taking a radical. Nonetheless, the full ideal is still not understood if $n\ge 4$.

In this article we turn from studying polynomial equalities related to tensor rank issues,  to inequalities.
Our main contribution is a generalization to arbitrary $n$ of the $n=2$ hyperdeterminant, $\Delta$, of Cayley. We obtain polynomial functions whose nonvanishing singles out a dense orbit of the tensors of rank $n$ from their closure.
We emphasize that this generalization does \emph{not} lead to those functions standardly called `hyperdeterminants' in the modern mathematics literature \cite{GKZ}, but rather to a set of functions that generalize the properties of $\Delta$ in another way, appropriate to the problem at hand. Just as $\Delta$ defines a one-dimensional representation of $GL(2,\mathbb C) \times GL(2,\mathbb C)\times GL(2,\mathbb C)$, our functions determine a multidimensional representation of  $GL(n,\mathbb C) \times GL(n,\mathbb C)\times GL(n,\mathbb C)$. 

\smallskip

The usefulness of studying invariant spaces of polynomials for investigating tensor rank issues is, of course, not new (see, for instance, the survey \cite{LandsbergSurvey}, and its references, for instances going beyond $\Delta$). Because the relevant representation theory is so highly developed, such study can be  fairly abstract. While the gap between abstractly understanding such polynomials and concretely evaluating them is conceptually a small one, in practice it is by no means trivially bridged (\cite{BatesOeding} gives an excellent illustration of this). 
Since our arguments depend crucially on being able to explicitly evaluate our invariant polynomials on certain tensors,
a concrete approach to developing them is warranted here.

Specifically, we investigate the transformation properties of our functions under a group action, a reduction under that action of most tensors to a semi-canonical form which is made possible by knowledge of the commutation relations,  and the evaluation of our functions at these semi-canonical forms. Together, these allow us to make precise statements about the zero set of these polynomials within the closure of the rank-$n$ $n\times n\times n$ tensors that are analogous to statements about the zero set of $\Delta$ in the $2\times2\times2$ case. 
After this initial development, we reframe our work in the language of representation theory. Finally
we show how our generalization of $\Delta$ can be used for three different applications.

\medskip
 This paper is organized as follows:
After definitions and preliminaries in \S \ref{sec:defs}, in \S \ref{sec:variety} we recall relevant facts about the algebraic variety of $n\times n\times n$ tensors of rank $n$, and construct semi-canonical orbit representatives under the group action. We then construct our invariant polynomial functions in \S \ref{sec:construct}, and determine on which complex tensors of complex border rank at most $n$ they vanish. Then in \S \ref{sec:reps} we study these functions in the framework of representation theory for the group
 $GL(n,\C)\times GL(n,\C)\times GL(n,\C)$.
 
 As a first application, in \S \ref{sec:real} we use these polynomials to investigate real tensors of complex rank $n$.  We show that the zero set of our invariants divides the 
real points on the variety into several path components, each of which contains tensors of a single \emph{signature}. These signatures are characterized by the number of complex conjugate pairs of rank-1 tensors in their unique rank-$n$ decomposition. We also determine the number of path components of each signature, and find that the tensors of real rank $n$ form a single component. We extend, from $n=2$ to $n=3$, the result that the sign of a polynomial invariant distinguishes whether a tensor of complex rank $n$ also has real rank $n$. For larger $n$ the sign of our invariant is insufficient to single out the component composed of tensors of real rank $n$, but why it fails to do so is made clear.

In \S \ref{sec:stat} we turn to the application which originally
motivated our interest in $n\times n\times n$ tensors of rank $n$, which is their appearance as certain statistical models, in both latent class analysis and phylogenetics. For these applications (which we introduce more thoroughly in \S \ref{sec:stat}), such a tensor represents a joint probability distribution of three observed random variables, each with discrete state space of size $n$, and thus has non-negative entries summing to 1. Its decomposition into a sum of rank-1 tensors reflects the structure of the stochastic model, in which the distributions of each of the observed variables depend on the state of a common hidden (latent, or unobservable) variable with $n$ states. The phylogenetic application can be seen through an interpretation of the  observed variables as  having $4$ states, the bases $A,C,G,T$ that may appear at a particular site in a DNA sequences from 3 species, while the state of the hidden variable represents the base in an ancestral organism from which the others evolved.

For these statistical applications, the rank-1 components of these tensors are themselves probability distributions, up to scaling, so it is important to not only determine the rank of a tensor, but also to be able to determine if the rank-1 components have non-negative entries. The  $2\times2\times 2$ rank-2 case and its extension to phylogenetic trees has recently been studied in \cite{ZS} (see also \cite{Klaere,PearlTarsi}). In that work Cayley's hyperdeterminant $\Delta$  played an important role. Our work here began as a step toward extending some of the results of \cite{ZS} from $n=2$ to $n>2$. In this paper, however, we limit ourselves to the simplest phylogenetic model (on a 3-leaf tree), as the extension to larger trees depends on other ideas which will be presented in \cite{ART}.
Despite lacking a good test for determining that a tensor of complex rank $n$ has real rank $n$, we borrow ideas from \cite{ART} to give semialgebraic conditions that ensure a tensor  is a probability distribution arising from the latent class model (with certain mild conditions on the parameters). 

As a final  application of the main theorems of this paper, in \S \ref{sec:jump} we show examples of $n\times n\times n$ tensors of  border rank $n$, but rank larger than $n$. We give a simple, explicit example of a tensor of the sort with rank $2n-1$. When $n=2$ this gives the well-known canonical form of a  complex rank 3 tensor; however, for general  $n>2$ it appears to be new class of examples of `rank jumping' by a large amount.

\section{Three-dimensional tensors,  group actions, and rank}\label{sec:defs}
Denote the space of all complex tensors of \emph{format} $(n_1,n_2,n_3)$ by $S=S(n_1,n_2,n_3)=\C^{n_1}\otimes \C^{n_2}\otimes\C^{n_3}$. Note $S\cong \C^{n_1n_2n_3}$, but that one may view an element of $S$ concretely as a $n_1\times n_2 \times n_3$ array of complex numbers. We thus identify such tensors with three-dimensional hypermatrices.

Let  $$G(\C)=GL(n_1,\C)\times GL(n_2,\C)\times GL(n_3,\C),$$ and $$G(\R)=GL(n_1,\R)\times GL(n_2,\R)\times GL(n_3,\R)\subset G(\C),$$ which act on $S$ through the 3 indices of tensors.
We write this action on the right, using several interchangable notations, so that for $P\in S$, $(g_1,g_2,g_3)\in G(\C)$
$$P(g_1,g_2,g_3)=((P*_1 g_1)*_2g_2)*_3g_3=\cdots=((P*_3 g_3)*_2g_2)*_1g_1,$$
where, for instance,
$$(P*_3g_3)_{ijk}=\sum_{l=1}^{n_3} P_{ijl}g_3(l,k),$$
with similar formulas for the action in other indices.
This notation is also useful for multiplication by vectors, so that if $\mathbf v\in \C^{n_3}$, for instance, then $P*_3\mathbf v$ is a matrix with entries
$$(P*_3\mathbf v)_{ij}=\sum_{k=1}^{n_3} P_{ijk}v_k.$$

For $i\in\{1,2,3\}$ by the \emph{$i$-slices} of $P$ we mean the $n_i$ matrices $P*_i\mathbf e_j$, whose entries have the $i$-th index fixed as $j$.
For example for $i=3$, a tensor $P$ has matrix slices $P_1,P_2,\dots ,P_{n_3}$, where $P_j=P(\cdot,\cdot,j)$. The action of $G(\C)$ on a tensor can be understood through transformation of these slices, as $P'=P(g_1,g_2,I)$ has slices $P'_j=g_1^TP_jg_2$,
while $P''=P(I,I, g_3)$ has slices $P''_j=\sum_k P_k g_3(k,j)$.

\medskip

The \emph{complex tensor rank} of  $P\in S$ is the smallest integer $r$ such that
$$P=\sum_{i=1}^r \mathbf u_i\otimes\mathbf v_i\otimes \mathbf w_i,$$
for some $\mathbf u_i\in \C^{n_1}$, $\mathbf v_i\in \C^{n_2}$, and $\mathbf w_i\in \C^{n_3}$. For a real tensor, the real tensor rank is defined analgously, requiring that $\mathbf u_i,\mathbf v_i,\mathbf w_i$ be real. Note that  if a tensor is real, its real and complex tensor ranks need not be equal, though the complex tensor rank is an obvious lower bound for the real tensor rank.

There is also a notion of \emph{multilinear rank} of such a tensor, which is an ordered triple $(r_1,r_2,r_3)$. Here $r_i$ is the 
rank of the transformation 
\begin{align*}\C^{n_i}&\to \C^{n_j}\otimes\C^{n_k},\\
\mathbf v&\mapsto P*_i\mathbf v,
\end{align*}
associated to $P$, and thus is the
ordinary matrix rank of the $n_jn_k\times n_i$ flattening of $P$. The multilinear rank of a real tensor is thus independent of the choice of field $\R$ or $\C$, as the analogous fact holds for matrices.

Both tensor rank (which we often will refer to as simply \emph{rank}, or $\C$-rank or $\R$-rank if the field must be made clear) and multilinear rank are invariant under the action of the general linear groups. More precisely, if 
$P'=Pg$ with $g\in G(\C)$, then $P'$ and $P$ have the same $\C$-rank and multilinear rank. If $P$ is real and $g\in G(\R)$, then $P$ and $P'$ have the same $\R$-rank as well.

\medskip
For the remainder of the paper, we restrict our attention to the case $$n_1=n_2=n_3=n,$$ though, as pointed out in the introduction, many results are easily modified to cases with $n_i\ge n$.

If $\mathbf v\in \C^n$, let $\diag(\mathbf v)$ denote the $n\times n$ diagonal matrix whose $(i,i)$-entry is $v_i$. Similarly, let $\Diag(\mathbf v)$ denote the $n\times n\times n$ diagonal tensor with $(i,i,i)$-entry $v_i$.
In particular, if $\mathbf 1$ is the vector with all entries 1, then $\diag(\mathbf 1)=I_n$ is the identity matrix. We denote its tensor analog by $D=D_n=\Diag(\mathbf 1)$.

\smallskip

By $\mathcal D(\C)$ and $\mathcal D(\R)$ we denote the $G(\C)$- and $G(\R)$-orbits of $D$, respectively.

\smallskip

\begin{proposition}  $\mathcal D(\C)$ is the set of all $n\times n\times n$ complex tensors of $\C$-rank $n$ and multilinear rank $(n,n,n)$.

$\mathcal D(\R)$ is the set of all $n\times n\times n$ real tensors of $\R$-rank $n$ and multilinear rank $(n,n,n)$.
\end{proposition}

\begin{proof}
First,  observe that $D=\Diag(\mathbf 1)$ has $\C$- and $\R$-rank $n$, and multilinear rank $(n,n,n)$: The multilinear rank is clear since the  $n^2\times n$ flattenings of $D$ all have the standard basis vectors among their rows. The $\C$- and $\R$-ranks of $D$ are at most $n$, since $$D=\sum_{i=1}^n \mathbf e_i\otimes \mathbf e_i\otimes \mathbf e_i.$$ If $D$ had  $\C$- or $\R$-rank  $k<n$, 
so $D=\sum_{i=1}^k  \mathbf u_i\otimes \mathbf v_i\otimes \mathbf w_i$,
then $$I_n=D*_3\mathbf 1=\sum_{i=1}^k  (\mathbf w_i\cdot \mathbf 1)\mathbf u_i\otimes \mathbf v_i$$
would have matrix rank $<n$, which is absurd. 

That every element of these orbits has the stated tensor rank and multilinear rank is a consequence of their invariance under the group actions. 

To see that every complex tensor $P$ of rank $n$ and multilinear rank $(n,n,n)$ lies in $\mathcal D(\C)$, first write
$$P=\sum_{i=1}^n \mathbf u_i\otimes\mathbf v_i\otimes \mathbf w_i.$$ Then since  the first flattening of $P$ to a $n^2\times n$ matrix has rank $n$, the $\mathbf u_i$ must be independent, so the matrix $g_1$ with $i$th row $\mathbf u_i$ is in $GL(n,\C)$. By the same reasoning, the matrices $g_2, g_3$ with $i$th rows $\mathbf v_i,\mathbf w_i$ are in $GL(n,\C)$. Then one checks that $P=D(g_1,g_2,g_3)\in \mathcal D(\C)$. 
The same argument applies in the real case.
\end{proof}

\smallskip

Note that
not every tensor of $\C$-rank $n$ is in $\mathcal D(\C)$. For instance, the tensor $P$ with slices $I,0,0,\dots,0$ has tensor rank $n$, but is not in $\mathcal D(\C)$ since it has multilinear rank $(n,n,1)$. However, $P$ is in the closure of $\mathcal D(\C)$, since one can give a sequence $\{h_i\}$ of elements in $GL(n,\C)$ with $\lim h_i= \mathbf 1\mathbf  e_1^T$, and then
$\lim D (I,I,h_i)=P$. Similarly reasoning shows that any tensor of complex tensor rank $\le n$ is in the closure of $\mathcal D(\C)$, and that an analogous statement holds for $\mathcal D(\R)$.

Finally, note that the closures of $\mathcal D(\C)$ and $\mathcal D(\R)$ also contain tensors of  rank  $>n$. Indeed, the phenomenon that tensor rank may \emph{increase} when one takes a limit is a key difference from the matrix rank. We will return to this with some explicit examples
in \S \ref{sec:jump}.

\section{The variety of rank $n$ tensors, and certain orbit representatives}\label{sec:variety}

Let $V_n\subseteq \C^{n^3}$ denote the closure of $\mathcal D(\C)$, under either the Zariski or standard topology, as these give the same set. This is the smallest algebraic variety containing all tensors of $\C$-rank $n$ and multilinear rank $(n,n,n)$. (It is straightforward to see 
it is also the smallest  variety containing all tensors of $\C$-rank $n$.)
Since $V_n$ is the closure of a $G(\C)$-invariant set, $V_n$ is also $G(\C)$-invariant.  As mentioned in the preceding section, for all $n\ge 2$, $V_n$ contains both tensors of rank less than $n$ and tensors of rank greater than $n$. Tensors in {$V_n\smallsetminus V_{n-1}$ have rank $\ge n$, and are said to have \emph{border rank} $n$.

A key fact we will use is that for all $n$ some defining equations for the variety $V_n$ are known, those given by the \emph{commutation relations} \cite{AR03,ARgm}. (The essential idea behind these seems to have first appeared in \cite{Strassen}.)

\smallskip

\begin{proposition} \label{prop:commute}  The ideal $I(V_n)$ of polynomials vanishing on $V_n$ includes those obtained from entries of the following matrix equations, $i=1,2,3$; $1\le j<k\le n$:
\begin{equation}\label{eq:commute}
(P*_i\mathbf e_j)\adj (P*_i\mathbf v) (P*_i\mathbf e_k)-(P*_i\mathbf e_k)\adj (P*_i\mathbf v) (P*_i\mathbf e_j)=0,
\end{equation}
where $\mathbf v\in \C^n$ is an arbitrary vector, and `adj' denotes the classical adjoint matrix.
\end{proposition}

While the identity \eqref{eq:commute} still holds if the $\mathbf e_j,\mathbf e_k$ appearing in it are replaced by more general vectors, this only yields linear combinations of the identities above and thus no essentially new polynomials. 
Moreover, by treating $\mathbf v$ as a vector of indeterminates, and considering the coefficients of the monomials in $\mathbf v$ that result from equation \eqref{eq:commute}, one may list a finite set of polynomials in $P$ linearly spanning this set for all choices of $\mathbf v$.

In the case $n=2$,  these relations are all trivial (\emph{i.e.,} simplify to $0=0$). If $n=3$, these polynomials are known to generate $I(V_3)$ \cite{GSS}. For $n\ge 4$, it is known that additional polynomials are needed to generate $I(V_n)$, but
little is known about them. A reward offered by one of the authors (ESA) in 2007 for determining $I(V_4)$ has led to that question being called  `The Salmon Problem.' Currently, only set-theoretic defining polynomials have been determined \cite{Friedland,BatesOeding,FriedlandGross}.

\medskip

Though the orbit of $D$ is dense in $V_n$, additional orbits lie in $V_n$ as well.
Next we show that some of the $G(\C)$-orbits in $V_n$ have orbit representatives of a certain form. This semi-canonical form will be used for determining on which tensors  the functions constructed in the next section vanish.
 
\smallskip
 
\begin{definition} An  $n\times n\times n$ tensor $P$ is \emph{$i$-slice-non-singular} if there is some $\C$-linear combination of the $i$-slices that is non-singular, and \emph{slice-non-singular} if it is $i$-slice-non-singular for some $i\in\{1,2,3\}$. If $P$ is not $i$-slice-non-singular, we say it is \emph{$i$-slice-singular}. If $P$ is not slice-non-singular, we say it is \emph{slice-singular}.
\end{definition}

Note that $P$ is $i$-slice-non-singular if, and only if, one can act on $P$ in the $i$th index by an element of $GL(n,\C)$ to obtain a tensor with a non-singular $i$-slice. Thus the terms in the above definition all depend only on the $G(\C)$ orbit of $P$. To investigate orbits, we consider slice-singular and slice-non-singular ones separately.

Let $\mathbf x=(x_1,x_2,\dots, x_n)$ be a vector of indeterminates. Then $P$ is $i$-slice-singular precisely when
$h_i(P;\mathbf x)= \det(P*_i\mathbf x)$ is the zero polynomial in $\mathbf x$. Thus the $i$-slice-singular tensors form an algebraic variety, defined by setting equal to zero the coefficients of each $\mathbf x$-monomial in the expansion of $h_i$.

\smallskip 

\begin{proposition}\label{prop:orbitrep}
Suppose $P$ is $i$-slice-non-singular and for that $i$ satisfies the polynomials of equation \eqref{eq:commute} in Proposition \ref{prop:commute}. Then $P$ has a $G(\C)$-orbit representative with all $i$-slices upper triangular. 

Moreover,
if the matrix $Z$ whose columns are the diagonals of the $i$-slices of such a representative is non-singular, then $P\in \mathcal D(\C)$. If $Z$ is singular,
then an orbit representative exists for $P$ with upper triangular $i$-slices and at least one slice strictly upper triangular.
\end{proposition}

\smallskip

Theorem \ref{thm:fzero} below implies that the matrix $Z$ of this theorem is non-singular precisely when $P\in \mathcal D(\C)$.

\smallskip

\begin{proof}
For convenience, suppose $P$ is $3$-slice-non-singular, with $3$-slices $P_1$, $P_2,\dots$, $P_n$.
Then, passing to other elements in its $G(\C)$ orbit, we may first assume $P$ has a non-singular slice, and then that  it has an identity slice, say $P_1=I$. But then the commutation relations of Proposition \ref{prop:commute} with $\mathbf v=\mathbf e_1$ say that for any $1\le j,k\le n$,

$$P_j \adj(P_1) P_k-P_k\adj(P_1)P_j=0,$$
so $$P_jP_k=P_kP_j.$$
Since these parallel slices commute, they can be simultaneously upper-triangularized, by a unitary $g_1$.  Thus acting by $(g_1^T,g_1^{-1},I)\in G(\C)$, we may assume the slices $P_i$ are all upper-triangular.

Let $Z$ be the matrix whose columns are the diagonals of the slices, \emph{i.e.,} $Z(i,j)=P(i,i,j)$. 

Suppose first that $Z$ is non-singular. Then acting on $P$ by $(I,I,g_2)$ for appropriately chosen $g_2$ will preserve the upper triangular form of the slices, keep $P_1=I$, but make another slice, say $P_2$, have distinct entries on its diagonal: To see this, note that the action of $(I,I,g_2)$ on $P$  will send $Z$ to $Zg_2$. Choose some $Z'$ whose first column is $\mathbf 1$, second column has distinct entries, and is non-singular, and choose $g_2$ so $Zg_2=Z'$.
Then $P'=P(I,I,g_2)$ will have all upper triangular slices, with $\mathbf 1$ on the diagonal of $P_1'$ and distinct entries
on the diagonal of $P_2'$.  To see that in fact $P_1'=I$, for any fixed $i<j$ consider the row vector $\mathbf w_{ij}=P(i,j,\cdot)$, whose entries come from the strictly upper triangular entries of the 3-slices. Now there is some row vector $\mathbf a$ with $\mathbf w_{ij} =\mathbf a Z$. But since the first entry of $\mathbf w_{ij}$ is 0, and the first column of $Z$ is $\mathbf 1$, we see $0=\mathbf a\mathbf 1$. Thus $\mathbf w_{ij}g_2=\mathbf a Zg_2=\mathbf a Z'$ implies the first entry of $\mathbf w_{ij}g_2$ is also 0, since the first column of $Z'$ is $\mathbf 1$.

Assuming now that $P_2$ has distinct entries on its diagonal, it can be diagonalized by acting on $P$ by some $(g_3^T,g_3^{-1},I)$, without changing $P_1=I$ or the diagonal entries of $P_2$. But then the commutation of $P_2$ with all other slices shows they are also diagonal. Moreover, $Z$ being non-singular is equivalent to a statement that no non-zero linear combinations of the upper-triangular slices is nilpotent. This property is preserved by the action of $(g_3^T,g_3^{-1},I)$, and so the new
matrix $Z$ of diagonals is non-singular as well. Thus by a final action by $(I,I,Z^{-1})$, we obtain $D$.

\smallskip

If, on the other hand, $Z$ is singular when the $P_i$ are upper triangular, then there exits a $g_4\in GL(n,\C)$ with $Zg_4$ having a column of zeros.
Acting on $P$ by $(I,I,g_4)$ preserves the upper triangular form of the slices, but ensures one slice has zeros on the diagonal.
\end{proof}

\section{ Construction of invariant functions, and their behavior on $V_n$}\label{sec:construct}
In this section, for all $n>2$ we construct explicit polynomial and rational functions on the $n\times n\times n$ tensors, with invariance properties under $G(\C)$. When $n=2$ this construction gives Cayley's hyperdeterminant $\Delta$, though for larger $n$ it appears to have not been studied before. We then investigate the values these functions take on when restricted to $V_n$. Using Proposition \ref{prop:orbitrep}, the explicitness of our construction allows us to show that the non-vanishing of the functions distinguishes the orbit $\mathcal D(\C)$.

\medskip

Let $\mathbf x=(x_1,\dots, x_n)$ be a column vector of auxiliary indeterminates. For a $n\times n\times n$ tensor $P$, consider the following functions, for $i\in \{1,2,3\}$:
\begin{gather}h_i(P;\mathbf x)=\det(P*_i \mathbf x),\\
f_i(P;\mathbf x)= (-1)^{n-1} \det(H_{\mathbf x}(h_i(P; \mathbf x))).
\end{gather}
Here $\det M$  denotes the determinant of a matrix $M$, and
$H_\mathbf x$ is the Hessian operator on a scalar-valued function, giving the matrix of 2nd-order partial derivatives with respect to the indeterminates $\mathbf x$.

These  functions are polynomials in the entries of $P$ and $\mathbf x$,  homogeneous in each. Their degrees are:
\begin{gather}\deg_P(h_i)=n ,\ \deg_\mathbf x(h_i)= n,\\
\deg_P(f_i)=n^2 ,\ \deg_\mathbf x(f_i)= n(n-2).
\end{gather}

From the action of $G(\C)$ on tensors, the functions above inherit certain invariance properties.
If $P'=P(g_1,g_2,g_3)$, and $\{i,j,k\}=\{1,2,3\}$ then one sees
\begin{equation}h_i(P';\mathbf x)=\det(g_j)\det(g_k)h_i(P;g_i\mathbf x ).\label{eq:htrans}
\end{equation}
Since by the chain rule, 
$$H_\mathbf x \left (h_i(P';\mathbf x) \right )=\det(g_j)\det(g_k) g_i^T \left ( (H_\mathbf x h_i) (P;g_i \mathbf x )\right )g_i ,$$
taking determinants yields
\begin{equation}
f_i(P';\mathbf x)=\det(g_j)^n\det(g_k)^n\det(g_i)^2 f_i(P;g_i \mathbf x).\label{eq:ftrans}
\end{equation}

\medskip

\noindent
\textit {Remark.} When $n=2$, note that $f_i(P;\mathbf x)=f_i(P)$ is independent of $\mathbf x$, and can be seen to be independent of $i$ as well, by calculating its explicit formula. Moreover, $f_i(P)=\Delta(P)$ since
in this case our construction is exactly Schl\"afli's construction of the $2\times2 \times 2$ hyperdeterminant from the $2\times 2$ determinant: Since
$\det(P*_i\mathbf x)$ is a quadratic form when $n=2$, the determinant of the Hessian of this form is the same as the discriminant of the form. Schl\"afli's construction is typically presented using the discriminant  \cite{GKZ}, and that formulation then generalizes to yield (a multiple of) the hyperdeterminant for tensors of larger dimension ($2\times2\times2\times 2$, etc.). Our functions $f_i$ are a
different generalization of the construction, for which the format of the tensor is $n\times n\times n$, and does not yield hyperdeterminants.

\smallskip

\noindent
\textit {Remark.} For $n\ge3$, $f_i$ is not independent of the auxiliary indeterminates $\mathbf x$. In classical language,
such a function might be called a \emph{covariant} for the $i$th factor in $G(\C)$, or a \emph{concomitant} (see, for instance, \cite{GraceYoungbook,Littlewoodbook,Olverbook}).
Only in the case $n=2$ is $f_i$ an \emph{invariant} in the strict sense of the term
 (\emph{i.e.}, associated with a one-dimensional representation). In \S \ref{sec:reps} we will see $f_i$ is associated to a higher dimensional representation when $n>2$.

\medskip

We next use the function $f_i$ to obtain a semialgebraic description of the orbit $\mathcal D(\C)$. Since the vector $\mathbf x$ has indeterminate entries, by a statement that $f_i(P; \mathbf x)=0$ we mean that when $f_i$ is evaluated at $P$ the resulting polynomial in $\mathbf x$ is identically zero. Thus $f_i(P; \mathbf x)\ne 0$ means at least one coefficient of an $\mathbf x$-monomial is non-zero at $P$.

\smallskip

\begin{theorem}\label{thm:fzero} 
$P\in \mathcal D(\C)$ if, and only if, for some $i\in\{1,2,3\}$, $P$ satisfies the equations \eqref{eq:commute} and $f_i(P;\mathbf x)\ne 0$.
Moreover, if these conditions hold for one $i\in\{1,2,3\}$, then they hold for all.
\end{theorem}

\begin{proof}
Since $D*_i\mathbf x =\diag(x_1,x_2,\dots, x_n)$, one computes 
\begin{equation}h_i(D;\mathbf x)=x_1x_2\cdots x_n \label{eq:hiD}
\end{equation}
so
\begin{multline*}f_i(D;\mathbf x)=(-1)^{n-1}\times\\
\det 
{\small{\begin{pmatrix} 0 &x_3x_4x_5\cdots x_n&x_2x_4x_5\cdots x_n&\cdots&x_2x_3x_4\cdots x_{n-1}\\ 
x_3x_4x_5\cdots x_n&0& x_1x_4x_5\cdots x_n&\cdots&x_1x_3x_4\cdots x_{n-1}\\
\vdots& &\ddots&&\vdots\\
x_2x_3x_4\cdots x_{n-1}&x_1x_3x_4\cdots x_{n-1}&x_1x_2x_4\cdots x_{n-1}&\cdots&0
\end{pmatrix},}}
\end{multline*}
so
$$(x_1x_2\cdots x_n)^2f_i(D; \mathbf x)=(-1)^{n-1}\det\left ( (x_1x_2\dots x_n)\begin{pmatrix} 0&1&1&\cdots &1\\1&0&1&\cdots &1\\
\vdots & &\ddots& &\vdots\\
1&1&1&\cdots &0
\end{pmatrix}\right ),$$
hence
\begin{equation}\label{eq:fiD}
f_i(D; \mathbf x) 
=(n-1)
(x_1x_2\dots x_n)^{n-2}.
\end{equation}
The transformation formula \eqref{eq:ftrans} then implies that if $P\in \mathcal D(\C)$ then $f_i(P;\mathbf x)\ne 0$ for all $i$. That equations \eqref{eq:commute} hold when $P\in \mathcal D(\C)$ is stated in Proposition \ref{prop:commute}.

\smallskip

Conversely, suppose for some $i$ that $f_i(P;\mathbf x)\ne 0$ and the equations \eqref{eq:commute} hold. Then $P$ must be $i$-slice-non-singular, since $i$-slice-singularity means $h_i(P; \mathbf x)$ is the zero polynomial, which implies $f_i(P;\mathbf x)=0$. Thus Proposition
\ref{prop:orbitrep} applies, and we see $P$
is $G(\C)$-equivalent to
a tensor $P'$ with upper triangular slices in index $i$. If  such a $P'$ had a slice with diagonal $\mathbf 0$, then $\det(P'*_i\mathbf x)$ would be independent of one of the $x_i$, so its Hessian would have  a zero row (and column), implying $f_i(P';\mathbf x)$ is the zero polynomial. By the transformation property \eqref{eq:ftrans}, it would follow that $f_i(P;\mathbf x)=0$ as well. Thus the slices of $P'$ cannot have diagonals of $\mathbf 0$, and Proposition \ref{prop:orbitrep} thus shows $P\in \mathcal D(\C)$.
\end{proof}

\smallskip

Note that the above theorem is only concerned with values of the $f_i$ on the variety defined by equations \eqref{eq:commute}; it makes no statement about $f_i$ off this variety.

\smallskip
Since the variety defined by equations \eqref{eq:commute} is a supervariety of $V_n$, we immediately obtain the following.

\begin{corollary}\label{cor:f}
$P\in \mathcal D(\C)$ if, and only if, $P\in V_n$ and $f_i(P;\mathbf x)\ne 0$ for some (and hence all) $i\in\{1,2,3\}$.
\end{corollary}

\medskip

Equations \eqref{eq:hiD} and \eqref{eq:fiD} suggest consideration of the rational function
$$r_i(P;\mathbf x)=\frac {f_i(P;\mathbf x)}{(n-1)h_i(P; \mathbf x)^{n-2}},$$
which is defined on the $i$-slice-non-singular tensors, and satisfies
\begin{equation}r_i(D;\mathbf x)=1.\label{eq:rone}\end{equation}
Moreover, if $P'=P(g_1,g_2,g_3)$, then the transformation formulas \eqref{eq:htrans} and \eqref{eq:ftrans} yield
\begin{equation}
r_i(P';\mathbf x)=\det(g_1)^2 \det(g_2)^2 \det(g_3)^2r_i(P,g_i\mathbf x).\label{eq:rtrans}
\end{equation}
Since $r_i(D;\mathbf x)$ is independent of $\mathbf x$, equation \eqref{eq:rtrans} implies that $r_i(P;\mathbf x)$ is independent of $\mathbf x$ when $P\in\mathcal D(\C)$, and thus, by continuity, when $P\in V_n$.

\smallskip

\begin{corollary}\label{cor:r}
$P\in \mathcal D(\C)$ if, and only if, $P$ satisfies the equations \eqref{eq:commute} and $r_i(P;\mathbf x)$ is defined and non-zero for some (and hence all) $i\in\{1,2,3\}$.
\end{corollary}

In this statement the condition that $P$ satisfies the equations \eqref{eq:commute} can of course be replaced by $P\in V_n$, as in Corollary \ref{cor:f}.

\medskip

While it is tempting to hope that $r_i(P;\mathbf x)$ is independent of $\mathbf x$ for all $P$,  one can verify that this is not the case even when $n=3$.
Indeed, for the tensor
$$P=\left [ \begin{pmatrix} 1&0&0\\0&0&0\\0&0&0 \end{pmatrix},\ 
  \begin{pmatrix} 0&0&0\\0&1&0\\0&0&0 \end{pmatrix},\ 
   \begin{pmatrix} 0&1&0\\1&0&0\\0&0&1 \end{pmatrix} \right ],$$ we have that
   $$h_3(P;\mathbf x)=\det \begin{pmatrix} x&z&0\\z&y&0\\0&0&z\end{pmatrix}=xyz-z^3,$$
   and
   $$f_3(P; \mathbf x)=\det \begin{pmatrix} 0&z&\phantom{-}y\\z&0&\phantom{-}x\\y&x&-6z\end{pmatrix}=2xyz+6z^3.$$
Thus $r_i(P;\mathbf x)$ is neither independent of $\mathbf x$, nor a polynomial function of $\mathbf x$. This example can easily be modified for $n\ge 3$.

\medskip

Suppose now one had a polynomial, 
$F(P),$ satisfying 
\begin{equation}
F(P')= \det(g_1)^{2k} \det(g_2)^{2k} \det(g_3)^{2k}F(P),\label{eq:Ftrans}\end{equation}
when $P'=P(g_1,g_2,g_3)$. That is, suppose  $F(P)$ is an invariant  of \emph{weight $(2k,2k,2k)$} for $G(\C)$}.
Then provided $F(D)\ne 0$, we may normalize so that $F(D)=1$, and then observe that by their transformation formulas
\begin{equation}
r_i(P)^k=F(P), \text{ for $P\in \mathcal D(\C)$}\label{eq:reqF}
\end{equation}
and thus, by continuity, 
$$f_i(P; \mathbf x)^k=(n-1)^k h_i(P;\mathbf x)^{k(n-2)}  F(P)$$  for all $P\in V_n$.
This yields the following.

\begin{proposition} \label{prop:tangles} Let $F(P)$ be an invariant of weight $(2k,2k,2k)$ for $G(\C)$, such that $F(D)\ne 0$. Then Theorem \ref{thm:fzero} and Corollary \ref{cor:f} remain true if $f_i$ is replaced by the function
$$G_i(P;\mathbf x)= h_i(P;\mathbf x)^{k(n-2)}  F(P).$$
\end{proposition}

While it is relatively straightforward to investigate the existence of polynomial invariants with weights of the type required for  $F(P)$ for small $n$, it is less easy to give them explicitly. 
We discuss this further in the next section.

\section{Representations and $n\times n\times n$ tensors}\label{sec:reps}
The previous section took an explicit, constructive approach to defining the invariants $f_i$. Here we turn to a 
 more general understanding of the representation theory of $G(\C)$. In particular, the transformation formula \eqref{eq:ftrans} of $f_i$ and Proposition \ref{prop:tangles} indicate that studying all polynomials with good transformation properties under the group action might be useful. As a general background to the material in this section, we suggest \cite{FultonHarrisbook,goodman2003,macdonald1979}

 \subsection{Representations and decompositions}
 
 As we will be concerned only with complex representations of complex groups, we suppress mention of $\C$ in our notation in this section. 
 We also let $V\cong \mathbb{C}^n$ (which should not be confused with our use of of $V_n$ for an algebraic variety in other sections).
 
 Recall that a \emph{representation} of a group $G$ is a homomorphism $\rho: G\rightarrow GL(W)$. If $W$ has no proper subspaces that are invariant under the action of $G$, then $\rho$ is said to be \emph{irreducible}. In particular, the irreducible representations of the general linear group on $V\cong \mathbb{C}^n$ are well understood to be
\[
\rho_\lambda: GL(V) \cong GL(n,\mathbb{C}) \rightarrow GL(V^\lambda) \cong GL(n_\lambda,\mathbb{C}),
\]
where the $\rho_\lambda$ are labelled by integer partitions $\lambda = (\lambda_1,\lambda_2,\cdots, \lambda_\ell)$, with $\lambda_1\geq \lambda_2 \geq \ldots \ge\lambda_\ell > 0$.
We say $\lambda$ is a partition of $m$, and write  $\lambda \vdash m$ and $|\lambda|=m$, when $\sum_{i=1}^{\ell}\lambda_i=m$. We refer to $\ell$ as the \emph{depth} of $\lambda$.
The representing space $V^\lambda\cong \mathbb{C}^{n_\lambda}$ (also referred to as a $G$-module), of dimension $n_\lambda$, can be expressed using the Schur functors $V^{\lambda}=\mathbb{S}_{\lambda}\left(V\right)$, but it is usually simpler to avoid explicitly doing so, and instead work with the characters of the representations.  
The characters $\left\{\lambda\right\}=\tr\circ \rho_\lambda$ are given by the Schur functions $s_\lambda$ \cite{Littlewoodbook}, with
\[
\left\{\lambda\right\}(g)=s_{\lambda}(\xi),
\] 
where $\xi= (\xi_1,\xi_2,\cdots,\xi_n)$ are class parameters (eigenvalues) for $g$.
Crucially, the Schur functions can be defined, and their properties explored, in a combinatorial manner quite independently of their role as characters for the general linear group \cite{macdonald1979}.

The dimension $n_\lambda=\{\lambda\}(I)$ of the representation $\rho_\lambda$ can be calculated by the hook length formula, which counts the number of semi-standard tableaux of shape $\lambda$. For instance, the defining representation of $GL(V)$ is associated to the partition $(1)$, with $n_{(1)}=n$, so $\{1\}$ denotes its trace and $s_{(1)}(\xi)=\xi_1+\xi_2+\cdots+\xi_n$.

From two representations $ \rho_\lambda$ and $\rho_{\lambda'} $ of $GL(V)$, one can construct the tensor product representation $(\rho_\lambda \otimes \rho_{\lambda'}) (g) := \rho_\lambda (g) \otimes \rho_{\lambda'} (g)$, with character denoted $\{\lambda\}\otimes\{\lambda'\}$. While this representation may be reducible, its decomposition into irreducible representations of $GL(V)$ can be found using the pointwise product of Schur functions:
$$
(\{\lambda\}\otimes \{\lambda'\})(g) = s_{\lambda}(\xi)s_{\lambda'}(\xi) = \sum_{\alpha}c^{\alpha}_{\lambda\lambda'}s_{\alpha}(\xi)
$$
Here the multiplicities $c^\alpha_{\lambda\lambda'}$ are computable using the Littlewood-Richardson rule \cite{macdonald1979}, with, for instance, software such as \texttt{Schur} \cite{schur}.

The irreducible representations of $GL(n_1)\times GL(n_2)\times GL(n_3)$ are  tensor products of the irreducible representations of the $GL(n_i)$: 
\[\rho_{\lambda_1}\times \rho_{\lambda_2} \times \rho_{\lambda_3}: GL(n_1)\times GL(n_2)\times GL(n_3)\rightarrow GL(n_{\lambda_1})\times GL(n_{\lambda_2})\times GL(n_{\lambda_3})\] 
where $(\rho_{\lambda_1} \times \rho_{\lambda_2} \times \rho_{\lambda_3})(g_1,g_2,g_3):= \rho_{\lambda_1}(g_1) \otimes \rho_{\lambda_2}(g_2) \otimes \rho_{\lambda_3}(g_3)$.
We denote the character of this representation by $\{\lambda_1\}\times \{\lambda_2\}\times \{\lambda_3\}$, to distinguish it from the product of characters of the sort described in the last paragraph.

The transformation of tensors $P\in U\cong  V\otimes V\otimes V$ under elements $g=(g_1,g_2,g_3)$ of $G=G(\C)=GL(n,\C) \!\times\! GL(n,\C)\! \times\! GL(n,\C)$ gives a representation of $G$ with character $\{1\}\times \{1\}\times \{1\}$:
\[ \tr(g_1\otimes g_2\otimes g_3)= \tr(g_1)\tr(g_2)\tr(g_3)=
s_{(1)}(\xi)s_{(1)}(\xi')s_{(1)}(\xi'').
\]
Denote the space of homogeneous polynomials of degree $d$ in the components of tensors $P\in U$ by ${\mathbb C}{[}U{]}{}_d$. 
This space inherits an action of $G$ by
\[
 f \mapsto g\circ f,  
\]
where $g\circ f(P):=f(P(g_1,g_2,g_3))$, and hence forms a $G$-module.
By a standard argument, it is possible to identify ${\mathbb C}[U]_d\cong U^{(d)}=\mathbb{S}_{(d)}\left(U\right)$.
While  $\mathbb C[U]_d$ is usually not an irreducible $G$-module, through characters we can identify its decomposition into irreducible modules.
This is done by applying the corresponding Schur function plethysm, denoted by $\underline \otimes$, 
\begin{equation}
 \left (\{1\}\times \{1\}\times \{1\} \right) \underline{\otimes} \{d\}=\sum_{\sigma_1,\sigma_2,\sigma_3\vdash n}\gamma^{(d)}_{\sigma_1\sigma_2\sigma_3}\{\sigma_1\}\times \{\sigma_2\}\times \{\sigma_3\},\label{eq:3pleth}
\end{equation}
where the multiplicities $\gamma_{\sigma_1\sigma_2\sigma_3}^{(d)}$ can be calculated using standard group theory techniques implemented in software such as \texttt{Schur} \cite{schur}. (See below for an outline, and \cite{sumner2008} for a more complete explanation.)
In terms of $G$-modules, we can use (\ref{eq:3pleth}) to identify
\begin{equation}
\mathbb{C}\left[U\right]_d\cong \bigoplus_{\sigma_1,\sigma_2,\sigma_3\vdash n}\gamma^{(d)}_{\sigma_1\sigma_2\sigma_3} V^{\sigma_1}\otimes V^{\sigma_2}\otimes V^{\sigma_3}.
\end{equation}
The primary focus of this section is to relate the functions $f_i(P;\mathbf x)$ to this decomposition.

\smallskip

The coefficients in a plethysm formula such as equations \eqref{eq:3pleth} are the structure constants for Schur function ``inner'' products, denoted by $*$:
\[
\{\alpha \}\ast\{ \beta \} = \begin{cases} \sum_{\mu\vdash n}\gamma^\mu_{\alpha \beta}\{\mu\}, & \text{ if $|\alpha|=|\beta|=n$,} \\ 0, & \text{ otherwise},  \end{cases}
\]
where $\gamma^\mu_{\alpha \beta}$ is the multiplicity of the irreducible representation $\{\mu\}$ occurring in the decomposition of the tensor product of the irreducible representations $\{\alpha\}$ and $\{\beta\}$ in the symmetric group $\mathfrak{S}_n$. 
By linearity and associativity, we can similarly define 
\[
\{\alpha \}\ast\{ \beta \}\ast \cdots \ast\{\zeta \}
= \sum_{\mu \vdash n}  \gamma^\mu_{\alpha \beta \cdots \zeta }  \{\mu\},
\]
and the general plethysm of a product of defining representations is
\begin{equation}
(\{1\}\times \{1\} \times \cdots \times \{1\}) \underline{\otimes} \{\mu\} = \sum_{\alpha,\beta,\cdots,\zeta \, \vdash\, |\mu|} 
\gamma^\mu_{\alpha \beta \cdots \zeta} \{\alpha  \}\times  \{\beta \}\times \cdots\times \{ \zeta\}.\label{eq:pleth}
\end{equation}
Equation \eqref{eq:3pleth} is then the special case of a three-fold  product.

\smallskip

\noindent
\textit{Remark.} If in equation \eqref{eq:pleth} the characters $\{1\}$ are replaced by  $\{\rho\} ,\{\sigma \} ,\{\tau \}, \cdots$, then the expansion on the right-hand side of equation \eqref{eq:pleth} would be over the respective plethysms,  
$(\{\rho\}\underline{\otimes} \{\alpha  \}), (\{\sigma \}\underline{\otimes}\{\beta \}), (\{\tau \}\underline{\otimes}\{\mu \}), \cdots$. The simpler case of equation \eqref{eq:pleth} arises since $\{1\}\underline{\otimes} \{\alpha  \} \equiv \{\alpha  \}$,  $\{1\}\underline{\otimes} \{\beta  \} \equiv \{\beta  \}$, etc.

\smallskip

\noindent
\textit{Remark.} A familiar application of this theory is given by considering the $k\times k$ minors of an $n\times n$ matrix $A$.
Under the action of $GL(n)\times GL(n): A\mapsto g_1Ag_2^T$, for each integer $1\leq k\leq n$, the span of the $k\times k$ minors of $A$ is an invariant subspace of the homogeneous polynomials of degree $k$ in the entries of $A$.
In terms of Schur function characters, the minors must therefore appear in the decomposition of the plethysm
\begin{equation}
\left(\left\{1\right\}\times \left\{1\right\}\right)\underline{\otimes}\left\{k\right\}=\sum_{\alpha,\beta \vdash k}\gamma^{(k)}_{\alpha\beta}\left\{\alpha\right\}\times\left\{\beta\right\}.\label{eq:2pleth}
\end{equation}
Adopting the standard shorthand of using exponents to signify repeated integers in a partition,
$(k)$ and $(1^k)=(1,1,\dots,1)$ label the one-dimensional ``trivial'' and ``sign'' representations of the symmetric group $\mathfrak{S}_k$, respectively. The calculation $\gamma^{(k)}_{(1^k)(1^k)}=1$ follows immediately from the tensor product $sgn(\sigma)\otimes sgn(\sigma)=sgn(\sigma)sgn(\sigma)=1$ for all $\sigma\in \mathfrak{S}_k$.

Thus the character $\left\{1^k\right\}\times \left\{1^k\right\}$ appears exactly once in the above expansion. 
Via the hook length formula one finds $\dim\left (\rho_{(1^k)}\times \rho_{(1^k)}\right )=\dim(\rho_{(1^k)})^2={n \choose k}^{2}$, and
the associated irreducible subspace is confirmed to be that spanned by the $k\times k$ minors of $A$.

\subsection{One-dimensional representations}
As a first use of this viewpoint, we investigate polynomial invariants of $G(\C)$, that is, one-dimensional representations within the vector space of polynomial functions. Proposition \ref{prop:tangles} indicates that
any polynomial $F$ satisfying equation \eqref{eq:Ftrans} for any positive integer $k$ can be used in characterizing the points on $\mathcal D(\C)$ (provided $F(D)\ne 0$), so we seek to find such $F$.
 
 The one-dimensional representation $\det(g)^k$ of $GL(n)$ has character
$\{k^n\}$. 
Via calculations with \texttt{Schur} \cite{schur}, by decomposing $\mathbb C\left[U\right]$ for $n=2,3,4$, we find there are polynomials in the entries of $P$ transforming as $\{2^2\}\times \{2^2\}\times \{2^2\}$, $\{2^3\}\times \{2^3\}\times \{2^3\}$, $\{2^4\}\times \{2^4\}\times \{2^4\}$ of degree $d=4,6,8$ for $n=2,3,4$ respectively. Since the multiplicities of the representation in the decomposition are found to be
\[
1=\gamma^{(4)}_{(2^2)(2^2)(2^2)}=\gamma^{(6)}_{(2^3)(2^3)(2^3)}=\gamma^{(8)}_{(2^4)(2^4)(2^4)},
\]
the polynomial is uniquely determined up to scaling. Thus for $n=2,3,4$ we denote this function by $\tau_n$, and call it the \emph{$n$-tangle}, since Cayley's hyperdeterminant $\tau_2=\Delta$ is called the tangle in the physics literature, and all transform with weight $(2,2,2)$. (We have not yet fixed a choice of scaling for $\tau_3,\tau_4$, but will below.)
This progression stops with $n=4$, but for $n=5,6$ there are invariants transforming with weight $(3,3,3)$, which again occur with multiplicity 1, and so are unique up to scaling.
These results are summarized in Table \ref{table:multi}.
Beyond $n=7$, computations with \texttt{Schur} become prohibitive, although we verified that there are no weight (2,2,2) representations for  $8\le n\le 16$. 

\begin{table}[h]
\begin{center}
\begin{tabular}{c|ccc}
Tensor& & Weight \\
format & (2,2,2) & (3,3,3) & (4,4,4)\\
\hline
$2\times2\times 2$ & $1$ & $0$ & $1$ \\
$3\times3\times 3$ & $1$ & $1$ & $2$\\
$4\times4\times 4$ & $1$ & $1$ & $5$\\
$5\times 5\times 5$ & $0$ & $1$ & $6$\\
 $6\times 6\times 6$& $0$ & $1$\\
$7\times 7\times 7$ & $0$ & $0$
\end{tabular}
\end{center}
\caption{Multiplicities  $\gamma_{(k^n)(k^n)(k^n)}^{(nk)}$ of one-dimensional representations of weight (k,k,k) in the space of homogeneous polynomials of degree $nk$ in the entries of tensors of format $n\times n\times n$, as computed by \tt{Schur}.}\label{table:multi}
\end{table}

\medskip
While the existence of invariants with the desired transformation property for $n<6$ is now established, to use them in Proposition \ref{prop:tangles} requires that we also know they do not vanish at $D$. Just as this can be easily seen from the explicit formula \eqref{eq:hyper} for the tangle when $n=2$, to establish this in the cases $n=3,4$ we turn to an explicit construction of the invariant.

For $n=3$, let $\epsilon_{ijk}$ be the antisymmetric Levi-Civita tensor with $\epsilon_{123}=\epsilon_{231}=\epsilon_{321}=1$,  $\epsilon_{213}=\epsilon_{321}=\epsilon_{123}=-1$, and $\epsilon_{ijk}=0$ otherwise. 
Consider the degree 6 polynomial $\tau_3:U\rightarrow \mathbb{C}$ defined by
\begin{multline*}
\tau_3(P)=\sum_{1}^3P_{i_1i_2i_3}P_{j_1j_2j_3}P_{k_1k_2k_3}P_{l_1l_2l_3}P_{m_1m_2m_3}P_{n_1n_2n_3}\times\\\epsilon_{i_1j_1k_1}\epsilon_{j_2k_2l_2}\epsilon_{k_3l_3m_3}\epsilon_{l_1m_1n_1}\epsilon_{m_2n_2i_2}\epsilon_{n_3i_3j_3},\nonumber
\end{multline*}
where all 18 indices run from 1 to 3 in the sum. Since $\epsilon_{ijk}$ defines a one-dimensional representation of $GL(3)$ by
\begin{align}
\sum_{1\leq i',j',k' \leq 3}g(i,i')g(j,j')g(k,k')\epsilon_{i'j'k'}=\det(g)\epsilon_{ijk},\quad \text{ for }{g\in GL(3,\mathbb{C})},\nonumber
\end{align}
it is straightforward to check that
$$\tau_3(P(g_1,g_2,g_3))=\det(g_1)^2\det(g_2)^2\det(g_3)^2\tau_3(P)$$ for all $(g_1,g_2,g_3)\in G(\mathbb{C})$.

As is noted in \cite{sumner2006},  expanding $\tau_3$ as a polynomial yields 1152 terms, and thus it is not the zero polynomial. But we wish to establish the stronger statement that $\tau_3(D)\ne 0$.
Expressing $D$ in components $D_{ijk}=\delta_{ij}\delta_{jk}$, one first finds
\begin{align}
\tau_3(D)=\sum_{1\leq i,j,k,l,m,n\leq 3}\epsilon_{ijk}\epsilon_{jkl}\epsilon_{klm}\epsilon_{lmn}\epsilon_{mni}\epsilon_{nij}.\nonumber
\end{align}
Now the first factor in the summand, $\epsilon_{ijk}$, is zero unless $i,j,k\in\{1,2,3\}$ are distinct. Then the product of the first two factors is zero unless additionally $l=i$.  Considering the remaining four $\epsilon$ factors in this way, non-zero contributions also require $m=j$, and $n=k$. Thus
$$
\tau_3(D)=\sum_{i,j,k \text{ distinct} }\epsilon_{ijk}^2\epsilon_{jki}^2\epsilon_{kij}^2 =6.$$
This confirms both that $\tau_3$ is a non-zero polynomial and that it evaluates to a positive value on the diagonal tensor.

Similar considerations give the 4-tangle. With $\epsilon_{ijkl}$ denoting the sign of the permutation $(ijkl)$ when $i,j,k,l$ are distinct, and 0 otherwise, let
\begin{align}
\tau_4(P)=\sum_1^4 P_{i_1i_2i_3}P_{j_1j_2j_3}P_{k_1k_2k_3}P_{l_1l_2l_3}P_{m_1m_2m_3}P_{n_1n_2n_3}P_{r_1r_2r_3}P_{s_1s_2s_3}\times\nonumber\\
\hspace{6em}\epsilon_{i_1j_1k_1l_1}\epsilon_{m_1n_1r_1s_1}\epsilon_{i_2l_2m_2s_2}\epsilon_{j_2k_2n_2r_2}\epsilon_{i_2j_3m_3n_3}\epsilon_{k_3l_3r_3s_3},\nonumber
\end{align} where all
24 indices run from 1 to 4 in the sum.
A polynomial expansion of this has 431,424 terms.
By an argument analogous to that for $\tau_3$, one sees that the only nonzero terms in
$$\tau_4(D)=\sum_1^ 4 \epsilon_{ijkl}\epsilon_{mnrs}\epsilon_{ilms}\epsilon_{jknr}\epsilon_{ijmn}\epsilon_{klrs}$$
occur when $m=k,n=l,r=i,s=j$,
and thus that
$\tau_4(D)=24$,
confirming that $\tau_4(D)$ is also non-zero.

We do not have an explicit construction of invariants for $n=5,6$ that have weight $(3,3,3)$, and we therefore do not know whether they vanish at $D$.

\subsection{Higher-dimensional representations}

The functions $f_i$ constructed in \S \ref{sec:construct} are not invariants when $n>2$, due to their dependence on the auxiliary variables $\mathbf x$, and thus do not define one-dimensional representations of $G(\C)$.
We therefore turn to studying higher dimensional representations in polynomials.

A consequence of considering irreducible modules of polynomials is that statements concerning polynomials vanishing on $G(\C)$-invariant sets which apply to  a specific element of the module must also apply to the module as a whole.                                                                                    
This is formalized in the two following lemmas.

\begin{lemma}\label{lem:Gspan}
Let $f\in\C[U]_d$. Then $\left \langle \{ g\circ f: g\in G\} \right \rangle_\C,$ the linear span of the $G$-orbit of $f$, is a $G$-module.
In particular, if $W$ is an irreducible submodule of $\C[U]_d$, and $0\ne f\in W$, then $\left \langle \{ g\circ f: g\in G\} \right \rangle_\C=W.$
\end{lemma}

\begin{proof} If $p\in\left \langle \{ g\circ f: g\in G\} \right \rangle_\C$, then, for some finite subset $S\subset G$ and $c_h\in \mathbb{C}$,
$$p=\sum_{h\in S}c_h\left(h\circ f\right),
$$
Thus if $g\in G$,
\[
g\circ p=\sum_{h\in S}c_h\left(gh\circ f\right)\in \left \langle \{ g\circ f: g\in G\} \right \rangle_\C,
\]so the linear span of the orbit is a $G$-module.

Now for any $G$-module $W$ and $f\in W$, $\left \langle \{ g\circ f: g\in G\} \right \rangle_\C\subseteq W.$ Irreducibility of $W$ and $f\ne0$ thus implies
$\left \langle \{ g\circ f: g\in G\} \right \rangle_\C=W$.
\end{proof}

\smallskip

Although as stated here this lemma applies to $f$ in a $G$-module of polynomials, the result is a standard one for any $G$-module.
Though also stated for polynomials, the next result holds more generally for $G$-modules of functions where the action of $G$ arises from an action on their domain.
\smallskip

\begin{lemma}\label{lem:Gvanish}
Let $\mathcal S$ be a $G$-invariant subset of $U$, and $W\subseteq \C[U]_d$ an irreducible $G$-module. 
Then $f|_\mathcal S\equiv 0$ for some non-zero $f\in W$ if, and only if, $p|_\mathcal S\equiv 0$ for all $p\in W$ .
\end{lemma}
\begin{proof}
Let  $f\in W$ with $f|_\mathcal S\equiv 0$. Then for any $p\in W$, $P\in U$, by the preceding lemma
$$p(P)=\sum_{h\in S}c_h\left(h\circ f\right)(P)=\sum_{h\in S}c_h\,f(Ph).$$
But $P\in \mathcal S$ implies $Ph\in\mathcal S$, so this shows $p|_\mathcal S\equiv 0$.
\end{proof}

\smallskip

There are several $G$-invariant sets of interest in this paper. They are $\mathcal D(\C)$, the orbit of the tensor $D$; $V_n=\overline{\mathcal D(\C)}$, the orbit closure; and $V_n\smallsetminus \mathcal D(\C)$, the complement of the orbit in its closure. However, by continuity, polynomials that vanish on $\mathcal D(\C)$ vanish on its closure $V_n$ as well, so investigating  polynomials vanishing on either of these sets  leads to polynomials in the defining ideal of the variety $V_n$.  Indeed, some such polynomials are given in Proposition $\ref{prop:commute}$, though not from the point of view of representations.

The functions $f_i(P;\mathbf x)$ constructed in the \S \ref{sec:construct}, however, are zero on $V_n\smallsetminus \mathcal D(\C)$, and
non-zero on $\mathcal D(\C)$, by  Corollary \ref{cor:f}. Thus Lemma \ref{lem:Gvanish} suggests relating the classical viewpoint on $f_i$ used in their construction to the
language of representations.
Without loss of generality we focus on $f_3(P;\mathbf x)$ and think of $f_3(P;\mathbf x)$ as providing a set of functions $f_3(\, \cdot\, ;\mathbf x): P \mapsto f_3(P;\mathbf{x})$ parametrized by the auxiliary variables $\mathbf x$.

Notationally, we use non-negative integer vectors $\alpha=(\alpha_1,\alpha_2,\ldots ,\alpha_n)$  to express a  monomial  $x^\alpha:=x_1^{\alpha_1}x_2^{\alpha_2}\ldots x_n^{\alpha_n}$ of total degree $d=\sum_{1\leq i\leq n}\alpha_i$.
Then
\[f_{3}(P;\mathbf x)=\sum_{\alpha}p_{\alpha}(P)x^{\alpha},\]
with coefficients $p_{\alpha}\in \mathbb{C}[U]_{n^2}$, where the monomials $x^\alpha$ are interpreted as basis elements for $\mathbb{C}[V]_d\cong V^{(d)}$, with $d=n(n-2)$.

One can see directly that no $p_\alpha$ is identically zero: To start, note that equation \eqref{eq:fiD} gives
\begin{align}
\sum_\alpha p_{\alpha}(D(I,I,g_3)) x^\alpha&=f_3(D(I,I,g_3);\mathbf x)\\
&=\det(g_3)^2f_3(D;g_3\mathbf x)\nonumber\\
&=\det(g_3)^2(n-1)((g_3 \mathbf x)_1(g_3 \mathbf x)_2\ldots (g_3 \mathbf x)_n)^{n-2}.\nonumber
\end{align}
Choosing $g_3\in GL(n,\mathbb{C})$ with strictly positive entries, every possible monomial $x^\alpha$ appears in the expansion of $((g_3 \mathbf x)_1(g_3 \mathbf x)_2\ldots (g_3 \mathbf x)_n)^{n-2}$. Hence $p_{\alpha}(D(I,I,g_3))\neq 0$ for all $\alpha$.

To understand the transformation of $p_\alpha$ under  $g=(g_1,g_2,g_3)\in G(\C)$, observe that $g_3$ maps the monomial $x^\alpha$ of total degree $d$ by 
\begin{align}
x^\alpha\mapsto  &\left(\sum_{i_1=1}^ng_3(1,i_1)x_{i_1}\right)^{\alpha_1}\left(\sum_{i_2=1}^ng_3(2,i_2)x_{i_2}\right)^{\alpha_2}\ldots \left(\sum_{i_n=1}^ng_3(2,i_n)x_{i_n}\right)^{\alpha_n}\nonumber\\
&:=\sum_{\beta}\tilde{g}_3(\alpha,\beta)x^{\beta}.\nonumber
\end{align}
The matrix elements $\tilde{g}_3(\alpha,\beta)$ provide precisely the irreducible representation $\rho_\lambda:GL(n,\mathbb{C})\rightarrow GL(n_\lambda,\mathbb{C})$ with $\lambda=(d)$ and $\tilde{g}_3=\rho_{(d)}(g_3)$.
The polynomials $p_{\alpha}(P)$ therefore transform under $G(\C)$ as
\begin{equation}
p_{\alpha}\mapsto \det(g_1)^n\det(g_2)^n\det(g_3)^2\sum_{\beta}p_\beta \tilde{g}_3(\beta,\alpha).\label{eq:palpha}
\end{equation}
This formula also implies that the $p_\alpha$ are independent: If $\mathbf c=(c_\alpha)$ specifies a dependency relation $\sum c_\alpha p_\alpha=0$,
then from \eqref{eq:palpha} it follows that $\mathbf d= \tilde{g}_3\mathbf c$ gives another dependency relation for every choice of $g_3$. By the irreducibility of $\rho_{(d)}$, and varying $g_3$, this can happen only if all $p_\alpha$ vanish identically, which they do not.

\smallskip

Now let $W$ be the span of  $\{p_\alpha\}$. Since the $p_\alpha$ are independent,  \eqref{eq:palpha} defines a linear map on $W$, making $W$ an irreducible $G(\C)$-module. 
The character of the corresponding representation of $G(\mathbb{C})$ is the product $\{n^n\}\times \{n^n\}\times \left(\{2^n\}\otimes \{d\}\right)$, with  $d\!=\!n(n-2)$.
Application of the Littlewood-Richardson rule \cite{macdonald1979} shows that, as a character of $GL(n,\mathbb{C})$ where partitions of depth greater than $n$ are excluded, the third factor is $\{2^n\}\otimes \{d\}=\{2+n(n-2),2^{n-1}\}$.

Thus the polynomial $f_3(P;\mathbf x)$ is associated to a module $W$ of polynomials in the entries of $P$ alone that transforms as $\{n^n\}\times \{n^n\}\times \{2+n(n-2),2^{n-1}\}$.
The dimension of such  an irreducible module is calculated by the hook length formula as
\begin{align*}
\dim({(n^n)})\times\dim((n^n))\times\dim({(2+n(n-2),2^{n-1})})&=1\times 1\times \binom{n(n-2)+(n-1)}{n(n-2)}\\
&=\binom{n^2-n-1}{n-1}.
\end{align*}
This result is not a surprise, since it is the dimension of the space of homogeneous polynomials in $n$ variables of degree $n(n-2)$,\emph{ i.e.}, the cardinality of the basis $\{p_\alpha\}$ of $W$.

\medskip

The multiplicity of modules transforming as  $\{n^n\}\times \{n^n\}\times \{2+n(n-2),2^{n-1}\}$ in the decomposition of $\mathbb{C}\left[U\right]_{n^2}$ can be calculated, at least for a few small values of $n$, using \texttt{Schur}, and are given in Table~\ref{tb:decomp}. 
Note that in the notation for the multiplicity $$\gamma^{(n^2)}_{(n^n)(n^n)(2+n(n-2),2^{(n-1)}),}
$$ whose values are given in the table,
 the superscript ${(n^2)}$ should be read as a 1-part partition of the integer $n^2$, while the subscripts $(n^n)$ denote the $n$-part partition $(n,n,\ldots,n)$ of $n^2$ in the standard shorthand notation for partitions.

\begin{table}[h]
\begin{center}
\begin{tabular}{l|lll}
$n$ & Module & Multiplicity & Dimension \\
 \hline
$2$ &  $\{2^2\}\times \{2^2\}\times \{2^2\}$ & 1 & 1\\
$3$ &  $\{3^3\}\times \{3^3\}\times \{5,2^2\}$ & 2 & 10 \\
$4$ &  $\{4^4\}\times \{4^4\}\times \{10,2^3\}$ & 5 & 165 \\
$5$ &  $\{5^5\}\times \{5^5\}\times \{17,2^4\}$ & 10 & 3876 \nonumber
\end{tabular}
\end{center}
\caption{Irreducible modules associated with the functions $f_3$ on $n\times n\times n$ tensors, along with their multiplicities $
\gamma^{(n^2)}_{(n^n)(n^n)(2+n(n-2),2^{(n-1)})}$ and dimensions $\binom{n^2-n-1}{n-1}$, 
 in the space of homogeneous polynomials of degree $n^2$ in the entries of $P$.}\label{tb:decomp}
\end{table}

For $n=3$, the multiplicity of 2 shown in Table \ref{tb:decomp} allows not only for the existence of $f_i$, but also for an additional function transforming in the same way.
Indeed, from the transformation formula \eqref{eq:htrans} for $h_i$ and the fact that $\tau_3$ is an invariant of weight $(2,2,2)$,  
the function $G_i(P;\mathbf x)=h_i(P;\mathbf x)\tau_3(P)$ will be such a function. This is precisely the construction given in
Proposition \ref{prop:tangles}, and since $\tau_3(D)\ne 0$, that result applies. By the discussion preceding that proposition, $f_i/h_i$ is not independent of $\mathbf x$, and thus cannot be a multiple of $\tau_3$. Thus $f_i$ and $G_i$ are independent. 

Similarly, for $n=4$, $G_i=h_i^2\tau_4$ and $f_i$ transform in the same way, and can be seen to be independent. While Table \ref{tb:decomp} indicates the existence of 3 other independent functions with the same transformation property, we have no explicit understanding of them, and do not know what, if anything, they indicate about tensor rank.

\section{Application to real tensors}\label{sec:real}

In this section we investigate real tensors in $\mathcal D(\C)$. 
As Corollary \ref{cor:f} gives a semialgebraic description of $\mathcal D(\C)$, it is natural to seek a similar description of  $\mathcal D(\R)$.
To obtain conditions that define $\mathcal D(\R)$ we should of course include the additional condition that  a tensor $P$ be real. However, even in the case that $n=2$ this is not sufficient to define $\mathcal D(\R)$; one also needs that $\Delta(P)>0$ \cite{dSL}.

Using our functions $f_i$ as a tool, our main results are as follows.
For $n=3$ the sign of an invariant function can be used to distinguish $\mathcal D(\R)$, extending the $n=2$ result in this single case. For all $n\ge 2$, the zero set of our $f_i$ partitions the real points in $\mathcal D(\C)$ into connected components. Within a component, all tensors have the same number of complex conjugate pairs of  rank-1 tensors in their rank decompositions. Moreover, with one exception, there is a single component for each allowable number of pairs. In particular one component
of the real points is $\mathcal D(\R)$.

\medskip

Let $V_n(\R)=V_n\cap \R^{n^3}$ denote the real points on $V_n$.
\smallskip

\begin{lemma}\label{lem:real}
Suppose $P\in\mathcal D(\C)\cap V_n(\R)$. Then, up to simultaneous permutation of the rows of the $g_i$, $P$ can be uniquely expressed as $P=D(g_1,g_2,g_3)$, subject to the following conditions: 
\begin{romannum}
\item The first non-zero entry of every row of $g_1$ and $g_2$ is 1,
\item For some  $k\le n/2$,  the first $2k$ rows of every $g_i$ are complex (and neither real nor purely imaginary), in conjugate pairs, and the remaining rows are real.
\end{romannum}
Thus $P$ has a unique decomposition into complex rank-1 components with $2k$ complex components, in conjugate pairs, and $n-2k$ real components.
\end{lemma}

\begin{proof}
For some $g_1,g_2,g_3\in GL(n,\C)$, $P=D(g_1,g_2,g_3)$, which implies
\begin{equation}
P=\sum_{i=1}^n \mathbf g_1^{i} \otimes \mathbf g_2^{i}\otimes \mathbf g_3^i, \label{eq:tendec}
\end{equation}
 where $\mathbf g_j^i$ is the $i$th row of $g_j$.
Since the rows of each $g_j$ are independent, Kruskal's Theorem \cite{Kruskal,Kruskal2,RhodesKrusk} implies this rank-1 decomposition is unique, up to ordering of the summands.
However the individual vectors  $\mathbf g_j^i$ can be multiplied by scalars $a_j^i$, as long
as $a_1^ia_2^ia_3^i=1$. Requiring the first non-zero entries in each row of $g_1, g_2$ to be 1 removes that freedom. 

Since $P$ is real, the complex conjugate of the decomposition in equation \eqref{eq:tendec} must give the same decomposition, up to order. Thus for each $i$
either $ \mathbf g_1^{i} \otimes \mathbf g_2^{i}\otimes \mathbf g_3^i$ is real, or its complex conjugate also appears as a summand. 

We may thus simultaneously permute the rows of the $g_j$ so for $i=1,\dots, k$ 
$$\mathbf g_1^{2i-1} \otimes \mathbf g_2^{2i-1}\otimes \mathbf g_3^{2i-1}= \overline{\mathbf  g}_1^{2i} \otimes \overline{\mathbf g}_2^{2i}\otimes \overline{\mathbf g}_3^{2i},$$
and for $2k<i\le n$ the summand is real.

Having done this, since each $\mathbf g_1^i, \mathbf g_2^i$ has an entry of 1, from the conjugate summands we first see that $\mathbf g_3^{2i-1}= \overline{\mathbf g}_3^{2i}$ for $i=1,\dots,k$ and then that similar statements hold for the rows of $g_1$ and $g_2$. Thus all three $g_i$ have the first $2k$ rows in conjugate pairs. Moreover, none of these first $2k$ rows of any $g_i$ is real, lest there be a repeated row, contradicting that $g_i\in GL(n,\C)$. Likewise, these rows are not purely imaginary. 

That the remaining rows of the $g_i$ are real follows by an analogous argument.
\end{proof}

\medskip

For $P\in \mathcal D(\C)\cap V_n(\R)$, we
refer to the ordered pair $(n-2k,k)$ of this lemma as the \emph{signature} of $P$. For $P\notin D(\C)\cap V_n(\R)$, the rank-1 tensor decomposition may not be unique; so we leave the signature undefined. The orbit $\mathcal D(\R)$ thus comprises those  $P\in \mathcal D(\C)$ with signature $(n,0)$.

Information on the signature of a tensor can be obtained from the value of the function $r_i$ defined in section \ref{sec:construct}, as we now show.
\medskip

\begin{theorem}\label{thm:fifact}
For all $n$, the set of  tensors $P\in V_n(\R)$  for which $r_i(P)>0$ is precisely those tensors of signature $(n-2k,k)$ with $k$ even.

In particular,
for $n=2$,  $\mathcal D(\R)$ is precisely the set of  tensors $P\in V_2(\R)$   with $\Delta(P)>0$, and for  $n=3$, $\mathcal D(\R)$ is precisely the set of  tensors $P\in V_3(\R)$   with $\tau_3(P)>0$.

For any $n\ge 3$, $\mathcal D(\R)$ is precisely the set of  tensors $P\in V_n(\R)$  with $f_i(P; \mathbf x)\ne 0$ such that when $f_i(P; \mathbf x)$ is factored into linear forms as 
$$f_i(P; \mathbf x)=c\prod_{j=1}^nl_j(\mathbf x)^{n-2},$$
then all of the linear forms $l_j$ may be taken to be real.
\end{theorem}

In this statement the conditions $P\in V_n(\R)$  with $f_i(P; \mathbf x)\ne 0$ can be replaced, by Theorem \ref{thm:fzero} and Corollary \ref{cor:f}, by an explicit collection of polynomial equalities and
$f_i(P; \mathbf x)\ne 0$.

\begin{proof} If $P$ has signature $(n-2k,k)$, then $P=D(g_1,g_2,g_3)$ where each $g_i$ has exactly $k$ pairs of complex 
conjugate rows.  Thus $\overline g_i= \sigma g_i$ for some permutation $\sigma$ with $\det \sigma =(-1)^k$. This implies
$\overline {\det g_i}=(-1)^k\det g_i$, so $\det g_i$ is real or purely imaginary according to whether $k$ is even or odd. Thus
$(\det g_i)^2$ is positive or negative according to whether $k$ is even or odd.
Using the properties of $r_i$ given in formulae \eqref{eq:rtrans}  and \eqref{eq:rone} the first claim is established. 

Note that for $n=2,3$, the only allowable even value for $k$ is zero. The claim for $n=2$ is then immediate since $r_i=\Delta$ in that case. For $n=3$ the claim follows similarly, from  the fact that
$\tau_3$ is of weight $(2,2,2)$ and $\tau_3(D)>0$, so, by equation \eqref{eq:reqF}, $\tau_3$ is a positive multiple of $r_i$ when restricted to $\mathcal D(\C)$.

\smallskip

For arbitrary $n\ge 3$,  $P\in V_n$ and $f_i(P; \mathbf x)\ne 0$ is equivalent to $P\in \mathcal D(\C)$ by Corollary \ref{cor:f}.
By equations \eqref{eq:ftrans} and \eqref{eq:fiD},
we have that for $P=D(g_1,g_2,g_3)$,
$$f_i(P; \mathbf x)=c'  f_i(D,g_i\mathbf x)=c\prod_{j=1}^n l_j(\mathbf x)^{n-2},$$
for some scalars $c',c$, and linear forms $l_j$ defined by the rows of $g_i$.  If $P\in \mathcal D(\R)$, so the $g_i$ may be taken to be real, $f_i(P; \mathbf x)$ has a factorization
using real linear forms. For $n\ge 3$ and $P\notin\mathcal D(\R)$,  by Lemma \ref{lem:real} such a factorization exists with at least two $l_j$ complex and not rescalable to be real. By unique factorization in the ring $\C[\mathbf x]$, there can be no factorization into powers of real linear forms in this case.
\end{proof}

Note that the statement in this theorem about the factorization into linear forms of $f_i(P;\mathbf x)$ could be replaced with a similar one about $h_i(P;\mathbf x)$.

\medskip

We next consider the connected components obtained from $V_n(\R)$ by removing the zero set of an $f_i$.
Let $Z_n=\{P\in V_n(\R) ~|~ f_i(P; \mathbf x)=0\}$, and note that  $Z_n$ is independent of the choice of $i\in\{1,2,3\}$, by Corollary
\ref{cor:f}.
 
 \medskip
 
\begin{theorem}\label{thm:path}
On each path component of $V_n(\R)\smallsetminus Z_n$ the 
 signature is constant.
 For each  $0\le k<{n/2}$, there is 1 path component of signature $(n-2k,k)$. 
 When $n$ is even there are 4 components with signature $(0,n/2)$.
\end{theorem}

\smallskip

\begin{proof} 
Suppose a component of $V_n(\R)\smallsetminus Z_n$ contains tensors of signature $(n-2k,k)$ for two different values of $k$. Let $k(P)$ denote the second term in the signature $(n-2k,k)$ of a tensor $P$.
Then in this component we can choose a tensor $P_0$ and sequence of  tensors $P_1,P_2, P_3,\dots$ with 
$\lim_{\ell\to\infty}P_\ell=P_0$, $k(P_0)=k_0\ne k_1=k(P_\ell)$ for $\ell\ge 1$.
But then 
\begin{equation}\label{eq:lim1}
\lim_{\ell\to \infty} f_i(P_\ell; \mathbf x)=f_i(P_0; \mathbf x).
\end{equation}

By Theorem \ref{thm:fifact}  for each $\ell>0$ the function $f_i(P_\ell; \mathbf x)$ factors as
\begin{equation}\label{eq:fact1}
f_i(P_\ell; \mathbf x)=c_\ell\prod_{j=1}^nl_{\ell,j}(\mathbf x)^{n-2},
\end{equation}
where we assume the linear forms in this factorization have been normalized so that
$$l_{\ell,j}(\mathbf x)=\mathbf u_{\ell,j}\cdot \mathbf x,$$
with $||\mathbf u_{\ell,j} ||=1$ for all $j$.
We further assume  complex vectors $\mathbf u_{\ell,2j-1}=\overline{\mathbf u_{\ell,2j}}$ for $1\le j\le k_1$ are associated to the the non-real linear forms, and real vectors $\mathbf u_{\ell,j}$ $2k_1< j\le n$ to the real ones.

By compactness of the unit sphere in $\C^n$, passing to a subsequence of $\{P_\ell\}$, we may assume that for each $j$
$$\lim_{\ell\to \infty} \mathbf u_{\ell,j} =\mathbf u_j,$$ 
for some unit vector $\mathbf u_j$. Let $l_j(\mathbf x)=\mathbf u_j\cdot \mathbf x.$
Now by equations \eqref{eq:lim1} and \eqref{eq:fact1} we see
$$f_i(P_0; \mathbf x)= \left (\lim_{\ell \to \infty} c_\ell \right ) \prod_{j=1}^nl_j(\mathbf x)^{n-2}.$$
Since for $n-2k_1$ values of $j$ the $\mathbf u_{\ell,j}$ are real, at least this many of the $\mathbf u_j$ are. Thus, since $k_0\ne k_1$, we have $k_0<k_1$.

However, $k_0<k_1$ implies that  for some $1\le j\le k_1$, $\mathbf u_{2j}=\lim_{\ell\to \infty} \mathbf u_{\ell,2j} $ is real.
But since $\mathbf u_{\ell,2j-1}=\overline{\mathbf u_{\ell,2j}}$, this means $\mathbf u_{2j}=\mathbf u_{2j-1}$.
That is impossible, as these vectors are, up to scaling, rows of some $g_i$ where $P_0=D(g_1,g_2,g_3)$, and thus must be independent.

Thus the signature is  constant on each component.

\smallskip

The number of  path components  that exist for any fixed value of $k$, $0\le k\le n/2$, is always at least 1, since one can construct a real tensor of signature $(n-2k,k)$.  We now turn to giving an upper bound on the number of such components.

Note first that  a $2\times n$ matrix with complex conjugate rows can be expressed as 
$$\begin{pmatrix} 1 & i\\1&-i\end{pmatrix}\begin{pmatrix} \mathbf r_1\\\mathbf r_2\end{pmatrix},$$ 
for row vectors $\mathbf r_1, \mathbf r_2\in \R^n$. Thus if $J_k$ is an $n\times n$ block diagonal matrix with $k$ blocks of $\begin{pmatrix} 1 & i\\1&-i\end{pmatrix}$
and $n-2k$ singleton blocks of $1$,  then tensors in $V_n(\R)\smallsetminus Z_n$ with signature $(n-2k,k)$  form the $G(\R)$ orbit of $D(J_k,J_k,J_k)$. We thus seek to bound the number of  path components of this orbit.

Recall that $GL_n(\R)$ has two path components, $GL_n^+(\R)$ and $GL_n^-(\R)$, with membership according to the sign of the determinant. Then $G(\R)$ has 8 path components, one of which is $$G^+(\R)=GL_n^+(\R)\times GL_n^+(\R)\times GL_n^+(\R).$$ 
The trivial bound on the number of components of the $G(\R)$ orbit of $D(J_k,J_k,J_k)$ is thus 8.

Suppose $k\le (n-1)/2$, so tensors of signature $(n-2k,k)$ have  at least one real rank-1 component, and $J_k$ has at least one $1\times1$ diagonal block, which we assume is the last. Let $K=\diag(1,1,\dots, 1, -1)$, and observe that since $J_kK=K J_k$,

\begin{align*}D(J_k,J_k,J_k)(K,I,I)&=D(K,I,I)(J_k,J_k,J_k)\\&=D(I,K,I)(J_k,J_k,J_k)
\\&=D(J_k,J_k,J_k)(I,K,I),
\end{align*}
and similarly $D(J_k,J_k,J_k)(I,K,I)=D(J_k,J_k,J_k)(I,I,K).$ Since $\det(K)=-1$,
this implies that the $G(\R)$ orbit of $D(J_k,J_k,J_k)$ is the union of the $G^+(\R)$ orbits of $D(J_k,J_k,J_k)$ and $D(J_k,J_k,J_k)(K,I,I)$. To show there is only one path component, we now need only show $D(J_k,J_k,J_k)$ and $D(J_k,J_k,J_k)(K,I,I)$ are in the same component.

\smallskip

If $k<(n-1)/2$, then
$J_k$ has at least two $1\times1$ diagonal blocks in the last positions.
Let 
$$R_2=\begin{pmatrix} 0&1\\-1&0\end{pmatrix},\ \sigma_2=\begin{pmatrix} 0&1\\1&0\end{pmatrix}.$$
Then one checks that
$$D_2(\sigma_2,R_2,R_2)=D_2.$$
Letting $R$ be the $n\times n$ block diagonal matrix $R=\diag(1,1,\dots 1, R_2)$, and $\sigma=\diag(1,1,\dots,1,\sigma_2)$, it follows that
$$D(J_k,J_k,J_k)(\sigma,R,R)=D(J_k,J_k,J_k).$$
Since $\det(R)=1$ and $\det(\sigma)=-1=\det(K)$,  $D(J_k,J_k,J_k)(\sigma,R,R)$ is in the $G^+(\R)$ orbit of, and hence the path component of, $D(J_k,J_k,J_k)(K,I,I)$. Thus $D(J_k,J_k,J_k)$ and $D(J_k,J_k,J_k)(K,I,I)$ are in the same component, and there is only one path component of signature $(n-2k,k)$.

\smallskip

In the case when $n$ is odd and $k=(n-1)/2$, $J_k$ has a single $1\times1$ block in the last position. Observe
that
$$\begin{pmatrix} 1& i\\1 & -i\end{pmatrix} \begin{pmatrix} 1& 0\\0 & -1\end{pmatrix} =\begin{pmatrix} 0&1\\1 & 0\end{pmatrix}\begin{pmatrix} 1& i\\1 & -i\end{pmatrix},
$$
so if $L=\diag(1,1,\dots,1,-1,1),$ then $$J_kL=\diag(1,1,\dots,1,\sigma_2,1) J_k.$$ 
Thus 
$$D(J_k,J_k,J_k)(L,L,L)=D( J_k, J_k, J_k).$$
Now $D(J_k,J_k,J_k)(LK,LK,LK)$ is in the $G^+(\R)$ orbit of, and thus path component of, $D(J_k,J_k,J_k)$,  but
\begin{align*}
D(J_k,J_k,J_k)(LK,LK,LK)&=D(J_k,J_k,J_k)(L,L,L) (K,K,K)\\
&=D(J_k,J_k,J_k)(K,K,K)\\ &=D(J_k,J_k,J_k)(K,I,I).\end{align*}
Thus there is only one path component of signature $(1,k)$.

\smallskip

If $n$ is even and $k=n/2$,   $J_k$ has only $2\times2$ blocks on its diagonal. Then $J_k K= \sigma J_k$ where $\sigma=\diag(1,1,\dots,1,\sigma_2)$. Thus
\begin{align*}D(J_k,J_k,J_k)(K,K,I)&=D(\sigma,\sigma, I)(J_k,J_k,J_k)\\&=D(I,I,\sigma)(J_k,J_k,J_k)\\&=D(J_k,J_k,J_k)(I,I,K),
\end{align*}
with similar formulas for the action of $(K,I,K)$ and $(I,K,K)$ on $D(J_k,J_k,J_k)$. One then sees the $G(\R)$ orbit of $D(J_k,J_k,J_k)$ is the union of
the $G^+(\R)$ orbits of $D(J_k,J_k,J_k)$, $D(J_kK,J_k,J_k)$, $D(J_k,J_kK,J_k)$, and $D(J_k,J_k,J_kK)$. To show there are 4 path components of signature $(0,k)$ it remains to show these four tensors lie in different components.

To this goal, consider a point $P_0$ of signature $(0,k)$, so
$$P_0=D(J_k,J_k,J_k)(g_1,g_2,g_3),$$ with $(g_1,g_2,g_3)\in G(\R)$.
Since $f_i(P_0; \mathbf x)\ne 0$, we also have $h_i(P_0; \mathbf x)\ne 0$. But by equation \eqref{eq:htrans}, $$h_3(P_0; \mathbf x)=\det(g_1)\det(g_2)\det(J_k)^2 h_3(D(I,I,J_k);g_3\mathbf x)$$ 
while a direct calculation shows
$$h_3(D(I,I,J_k);\mathbf x)=\prod_{i=1}^{k}(x_{2i-1}^2+x_{2i}^2).$$ 
Thus $h_3(P_0;\mathbf x)$ is a non-zero polynomial in $\mathbf x$ whose values are either non-negative on all of $\R^n$, or non-positive, with similar statements valid for $h_1,h_2$. But a straightforward continuity argument shows that along a path composed of points $P$ with signature $(0,k)$ the polynomial $h_i(P;\mathbf x)$, viewed as a function of $\mathbf x$, cannot pass between being non-negative valued and non-positive valued without being identically zero for some $P$. Since it is not identically zero at any point of signature $(0,k)$, on each path component it must be either non-negative valued for all $P$, or non-positive valued for all $P$.

But since 
\begin{multline*}h_3(D(J_k,J_k,J_k);\mathbf x)= h_3(D(J_k,J_k,J_k)(I,I,K);\mathbf x)\\=-h_3(D(J_k,J_k,J_k)(K,I,I);\mathbf x) =-h_3(D(J_k,J_k,J_k)(I,K,I);\mathbf x)
\end{multline*}  and
\begin{multline*}
h_1(D(J_k,J_k,J_k);\mathbf x)= h_1D(J_k,J_k,J_k)(K,I,I);\mathbf x)\\ =-h_1(D(J_k,J_k,J_k)(I,K,I);\mathbf x)=-h_1(D(J_k,J_k,J_k)(I,I,K);\mathbf x),\end{multline*}
we can conclude that the 4 points all lie in different path components, and so there are 4 path components of points of signature $(0,k)$.
\end{proof}

\smallskip

\noindent
\emph{Remark.} Even in the well-studied case $n=2$, the assertion of Theorem \ref{thm:path} seems to be a new result.

\section{Application to a Stochastic Model}\label{sec:stat}

We consider here a discrete statistical model with a single hidden variable, in order to obtain a semialgebraic description of its set of probability distributions

\begin{figure}[h]
\begin{center}\includegraphics[height=1in,width=1in]{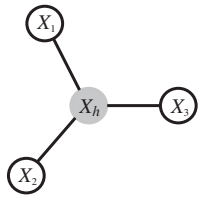}
\end{center}
\caption{A graphical model, in which the 3 leaves represent observed random variables, $X_1,X_2,X_3$, and the central node a hidden random variable, $X_h$. As considered here, all variables are assumed to have $n$ states. The structure of the graph indicates the leaf variables are independent when conditioned on the hidden one.}\label{fg:3taxa}
\end{figure}

 As a graphical model, it is specified by a 3-leaf tree, as shown in Figure \ref{fg:3taxa}. The internal node of the tree, and each leaf, represent random variables, all with $n$ states. The internal node variable is hidden (\emph{i.e.}, unobservable). The observed leaf variables are independent when conditioned on the state of the hidden one. The hidden variable thus provides an `explanation' of dependencies between the observed ones. 
This simple conditional independence model with a hidden variable is common in many statistical applications, and  variously called a hidden naive Bayes model, a latent class model, or a 3-leaf tree model, though often the four variables are allowed to have state spaces of different sizes. 

In applications, the observed variables in this model might represent three characteristics (such as the results, $+$ or $-$, of medical tests) measured on individuals in a population, while the hidden variable represents a `latent class' to which the individual belongs (such as whether the individual has or does not have a certain disease). The probabilities of the test outcomes depend on the disease condition, yet given an individual's disease state, the results of the tests are independent of each other.

Given this model for fixed $n$, one can view a probability distribution arising from it as an $n\times n\times n$ tensor. A natural question is to find a semialgebraic characterization of such tensors, that is, a collection of polynomial equalities and inequalities that precisely cut out the distributions arising from the model. The structure of the probability model ensures that the rank of the tensor is at most $n$, and moreover that each summand in a decomposition as a sum of $n$ rank-1 tensors has non-negative entries. Understanding polynomials equalities holding on these distributions amounts to understanding the defining ideal of the variety $V_n$, work on which was reviewed in \S \ref{sec:variety}.  Inequalities holding on such tensors are much more poorly understood.
While the existence of inequalities 
in the tensor entries that ensure it arises from meaningful stochastic parameters (e.g., non-negative) follows from general theory of real algebraic geometry, explicit inequalities have previously not been given for arbitrary $n$.

Our interest in the model is motivated by its appearance as the general Markov model in phylogenetics, where in the special case $n=4$ it is used to model evolution of DNA sequences by base substitution. One might think of the unobserved variable as representing the base (A, C, T, or G)  at a site in a sequence of an ancestral organism from which we have no data, and the observed variables as the state of the site in three currently extant descendants of it. Though trees with more than 3 leaves are of course essential for phylogenetic applications, in a related work \cite{ART} it is shown how to extend a semialgebraic description for the 3-leaf tree to $m$-leaf trees.

\smallskip

Denoting this model by $\mathcal M_n$, we first describe its parameterization. With $[n]=\{1,2,\dots, n\}$ as the state space of all random variables, let $\boldsymbol \pi=(\pi_1,\dots, \pi_n)$ denote the probability distribution vector for the hidden variable $X_h$, so $\pi_i=\PP(X_h=i)$. For the observed variables $X_i$, $i=1,2,3$, a $n\times n$ stochastic matrix $M_i$ has $(j,k)$-entry specifying the conditional probability $\PP(X_i=k\mid X_h=j)$, so each row of each $M_i$ sums to 1. 
The connection of this model to the tensor rank questions we study in this paper arises from the observation that a probability distribution for $\mathcal M_n$ is specified by the $n\times n\times n$ tensor 
$$P=\Diag(\boldsymbol \pi)(M_1,M_2,M_3).$$
The domain we consider for the parameterization map is specified by requiring 

\begin{enumerate}
\item $\boldsymbol \pi$ has strictly positive entries summing to 1, 
\item the $M_i$ have non-negative entries and row sums of 1, and
\item the $M_i$ are non-singular. 
\end{enumerate}
For some statistical applications it is also natural to strengthen the second requirement so that the $M_i$ have strictly positive entries; we thus comment on this situation as well in Proposition \ref{prop:stoch} below.

Note that a few trivial inequalities in the entries of $P$ that must hold  are obvious: Since $P\in \mathcal M_n$ is a probability distribution, its entries must be non-negative. If one additionally assumes the $M_i$ have strictly positive entries, then $P$ must have strictly positive entries as well.

 \medskip

For $n=2$, a complete semialgebraic description of $\mathcal M_2$ has been given in two recent independent works \cite{ZS, Klaere}, using different approaches. In particular, in \cite{ZS}  the $2\times 2\times 2$ hyperdeterminant $\Delta$ plays a key role, though many statistically-motivated ideas are also used. Here we give a semialgebraic description of $\mathcal M_n$ for all $n\ge 2$, using the invariants developed in earlier sections that generalize $\Delta$.

\smallskip

Recall a \emph{principal minor} of a matrix is the determinant of a submatrix chosen with the same row and column indices. A \emph{leading} principal minor is one for which these indices are $\{1,2,3\dots,k\}$ for some $k$.

\smallskip

\begin{proposition}\label{prop:stoch}
A $n\times n\times n$ tensor $P$ is in the image of the parametrization map for $\mathcal M_n$ if, and only if, the following conditions hold:
\begin{enumerate}
\item \label{it:1}$P$ is real, with non-negative entries summing to 1.
\item \label{it:2} For some (and hence all) $i$, $P$ satisfies the commutation relations given by equation \eqref{eq:commute}, and  the polynomial $f_i(P;\mathbf x)$ is not identically zero.
\item \label{it:3} $\det(P*_i\mathbf 1)\ne 0$ for all $i\in\{1,2,3\}$.

\item \label{it:5}For at least one (and hence all) of the following matrices, all leading principal minors are strictly positive:  
\begin{align}
\det(P*_1\mathbf 1) (P*_2 \mathbf 1)\adj(P*_1 \mathbf 1)(P*_3 \mathbf 1)^{T}\notag\\
\det(P*_2\mathbf 1) (P*_1\mathbf 1)\adj(P*_2 \mathbf 1)(P*_3 \mathbf 1)\label{eq:cond4}\\
\det(P*_3\mathbf 1) (P*_1 \mathbf 1)^T\adj(P*_3 \mathbf 1)(P*_2 \mathbf 1)\notag
\end{align}
\item \label{it:6} For all $1\le l \le n$, all principal minors of the three matrices
\begin{align}
\det(P*_1\mathbf 1) (P*_2 \mathbf 1)\adj(P*_1 \mathbf 1)(P*_3 \mathbf e_l)^{T}\notag\\
\det(P*_1\mathbf 1) (P*_2 \mathbf e_l)\adj(P*_1 \mathbf 1)(P*_3 \mathbf 1)^{T}\label{eq:cond5}\\
\det(P*_2\mathbf 1) (P*_1\mathbf e_l)\adj(P*_2 \mathbf 1)(P*_3 \mathbf 1)\notag
\end{align} 
are non-negative. 
\end{enumerate}
Here $\adj(M)$ denotes the classical adjoint of a matrix $M$.

If parameters of the model are restricted so that entries of $M_i$ are strictly positive, then in condition \eqref{it:6} one should replace `principal minors' by `leading principal minors' and  `non-negative' by `positive'.
\end{proposition}

\smallskip

Note that the only equality constraints in the theorem are those in conditions \eqref{it:1} and \eqref{it:2}. In particular, a full set of generators of the ideal
 $I(V_n)$ is not used (when $n\ge3$) in this semialgebraic description of the model.
 
\medskip

Our proof will use repeatedly the following well-known classical result on matrices defining quadratic forms. \begin{theorem}[Sylvester's Theorem]
Let $A$ be an $n\times n$ real symmetric matrix and $Q(\mathbf v)=\mathbf v^TA\mathbf v$ the associated quadratic form on $\R^n$.
Then
\begin{enumerate} 
\item $Q$ is positive definite if, and only if, all leading principal minors of $A$ are strictly positive.
\item $Q$ is positive semidefinite if, and only if, all principal minors of $A$ are non-negative.
\end{enumerate}
\end{theorem}

\medskip

\begin{proof}[of Proposition \ref{prop:stoch}]

We first discuss the necessity of these conditions. The necessity of \eqref{it:1} is clear. Condition \eqref{it:2} holds by Theorem \ref{thm:fzero}, since $$P=\Diag(\boldsymbol \pi)(M_1,M_2,M_3)=D(\diag(\boldsymbol \pi)M_1,M_2,M_3)$$ shows $P$ is in the $G(\C)$-orbit of $D$.

For \eqref{it:3}, observe that $P*_i\mathbf 1$ is the marginalization of the distribution to
two observed variables, so one sees
$$P*_i\mathbf 1=M_j^T\diag(\boldsymbol \pi)M_k$$
for distinct $i,j,k$. Since $\boldsymbol \pi$ has positive entries, and $M_j,M_k$ are non-singular, the determinant of this matrix is non-zero.

Condition \eqref{it:5} can be restated, after dividing the first formula of \eqref{eq:cond4} by the positive number $\det(P*_1\mathbf 1)^2$, as asserting the positivity of leading principal minors of
$$ (P*_2 \mathbf 1)(P*_1 \mathbf 1)^{-1}(P*_3 \mathbf 1)^{T},$$
and two similar expressions. But expressed in terms of parameters, this is 
$$(M_1^T\diag (\boldsymbol \pi)M_3 )(M_2^T \diag (\boldsymbol \pi)M_3)^{-1} (M_1^T \diag (\boldsymbol \pi)M_2)^T=
M_1^T\diag (\boldsymbol \pi)M_1.$$ This symmetric matrix, and similar ones obtained from the other expressions, define positive definite quadratic forms because $\boldsymbol \pi$ has positive entries. Thus by Sylvester's Theorem, all their principal minors are positive. 

A similar argument shows the necessity of condition \eqref{it:6}. For instance, letting $\mathbf r_l^3$ be the vector whose
entries are the products $\pi_i M_3(i,l)$, one sees
$$\det(P*_1\mathbf 1) (P*_2 \mathbf 1)\adj(P*_1 \mathbf 1)(P*_3 \mathbf e_l)^{T}=\det(P*_1\mathbf 1)^2 M_1^T \diag(\mathbf r_l^3) M_1.$$
Since the entries of $M_3$ are non-negative, this matrix defines a positive semidefinite quadratic form, and thus by Sylvester's Theorem has non-negative principal minors.
If the entries of the $M_i$ are positive, then this matrix  defines a positive definite form, and thus has positive leading principal minors.
\medskip

Turning to sufficiency, assume conditions (\ref{it:1}-\ref{it:6}) are met by a tensor $P$. By Theorem \ref{thm:fzero}, condition \eqref{it:2} 
implies $P=D(g_1,g_2,g_3)$ for some $g_i\in GL(n,\C)$. Moreover, by the realness of $P$ in condition \eqref{it:1}, from Lemma \ref{lem:real} we also know any complex entries in the $g_i$ occur in complex conjugate rows.
Our goal is to modify this expression, so the $g_i$ are replaced by stochastic matrices, and $D$ by a diagonal tensor with positive entries.

Letting $\mathbf s_i=g_i\mathbf 1$ be the vector of row sums of $g_i$, we have
\begin{align*}P*_1\mathbf 1&=g_2^T \diag(\mathbf s_1) g_3,\\
P*_2\mathbf 1&=g_1^T \diag(\mathbf s_2) g_3,\\
P*_3\mathbf 1&=g_1^T \diag(\mathbf s_3) g_2.
\end{align*}
Thus the non-vanishing of the determinants of these matrices by condition \eqref{it:3} tells us the row sums are all non-zero. Letting $M_i =\diag(\mathbf s_i)^{-1}g_i$, and $\boldsymbol \pi$ be the vector of entry-wise products of $\mathbf s_1, \mathbf s_2, \mathbf s_3$, we thus have
$P=\Diag(\boldsymbol \pi)( M_1, M_2, M_3)$. Here each $M_i$ has unit row sums,  $\boldsymbol \pi$ has non-zero entries, and
\begin{align*}P*_1\mathbf 1&=M_2^T \diag(\boldsymbol \pi) M_3,\\
P*_2\mathbf 1&=M_1^T \diag(\boldsymbol \pi) M_3,\\
P*_3\mathbf 1&=M_1^T \diag(\boldsymbol \pi) M_2.
\end{align*}
Since $P$ is real, these expressions are as well, though
we have not yet shown that $M_i$ and $\boldsymbol \pi$ have real entries. Nonetheless,  all $M_i$ have the same number of conjugate (non-real) pairs of rows, in corresponding positions, with the corresponding entries of $\boldsymbol \pi$ also conjugate (though possibly real).

Now substituting the above expressions for marginalizations in the three expressions in condition \eqref{it:5}, they simplify to
\begin{align*}\det(P*_1 \mathbf 1)^2&M_1^T\diag(\boldsymbol \pi) M_1,\\
\det(P*_2 \mathbf 1)^2&M_2^T \diag(\boldsymbol \pi) M_2,\\
\det(P*_2 \mathbf 1)^2&M_3^T \diag(\boldsymbol \pi) M_3.
\end{align*}
This shows that $M_i^T\diag(\boldsymbol \pi) M_i$ is real for each $i$. We now argue that if $M_i$ is not real, then the  quadratic form $Q_i$ associated to
$ M_i^T\diag(\boldsymbol \pi) M_i$ is not positive definite.
 To that end, suppose  two rows (say, the first two) of $M_i$ are complex conjugates, and thus by Lemma  \ref{lem:real}, of the form
 $$\mathbf m_i^1=\mathbf r_1 +i\mathbf r_2,\ \ \ \  \mathbf m_i^2=\overline{\mathbf m}_i^1=\mathbf r_1-i\mathbf r_2,\ \ \ \mathbf r_i\in \R^n \smallsetminus \{\mathbf 0\}$$ and the corresponding entries of $\boldsymbol \pi$ are $\pi_1,\pi_2=\overline{\pi}_1$.
 Then for any real vector $\mathbf v$ orthogonal to the real and imaginary parts of the other rows of $M_i$, evaluating the quadratic form at $\mathbf v$ yields 
 $$Q_i(\mathbf v)=\pi_1 (\mathbf m_i^1\cdot \mathbf v)^2 +\overline \pi_1 (\overline{\mathbf m}_i^1\cdot \mathbf v)^2.$$
 If we additionally choose $\mathbf v$ to be orthogonal to $\mathbf r_2$, but not to $\mathbf r_1$,
 then 
 $$Q_i(\mathbf v)=\pi_1 (\mathbf r_1 \cdot \mathbf v)^2 +\overline \pi_1 (\mathbf r_1\cdot \mathbf v)^2=2\Re(\pi_1) (\mathbf r_1\cdot \mathbf v)^2.$$
Positive definiteness of $Q_i$ would thus imply $\Re(\pi_1)>0$.
However, if we instead choose $\mathbf v$ to be orthogonal to $\mathbf r_1$, but not to $\mathbf r_2$,
 Then 
 $$Q_i(\mathbf v)=\pi_1 (i\mathbf r_2 \cdot \mathbf v)^2 +\overline \pi_1 (i\mathbf r_2\cdot \mathbf v)^2=-2\Re(\pi_1) (\mathbf r_2\cdot \mathbf v)^2,$$
so positive definiteness would imply $\Re(\pi_1)<0$. Thus if $M_i$ were not real, then $Q_i$ would not be positive definite.

But if $Q_i$ is not positive definite, by Sylvester's Theorem, the positivity of leading principal minors asserted in condition \eqref{it:5} must be violated. Thus condition \eqref{it:5} implies at least one of the $M_i$ is real, so all are by Lemma \ref{lem:real}. Applying Sylvester's theorem again to the positive definite form $Q_i$ then implies that $\boldsymbol \pi$ has real positive entries.

 Finally, condition \eqref{it:6} ensures the entries of the $M_i$ are non-negative. For instance
 \begin{equation}\det(P*_1\mathbf1)(P*_2 \mathbf 1)\adj (P*_1 \mathbf 1)(P*_3 \mathbf e_j)^{T}= \det(P*_1\mathbf1)^2 M_1^T\diag(\boldsymbol \pi)\diag(\tilde{\mathbf m}_3^j)M_1.\label{eq:quadform2}
\end{equation}
where $\tilde m_3^j$ is the $j$th column of $M_3$.
Thus all principal minors of this matrix being non-negative implies the associated quadratic form is positive semidefinite and thus that
$\tilde m_3^j$ has non-negative entries.
To instead ensure these entries are strictly positive, we require that the quadratic form be positive definite, and thus that all leading principal minors be positive.
\end{proof}

For further work in this direction, we direct the reader to \cite{ART}.

\section{Application to `rank jumping'}\label{sec:jump}

It is well known that the limit of a sequence of tensors of a fixed rank $r>1$ may be strictly larger than $r$ (that is, tensor rank, unlike matrix rank, is not upper semicontinuous).
This `rank jumping' is responsible for the fact that a given tensor may not have a best approximation by a tensor of fixed lower rank, and can thus be of concern in applied settings.

A tensor that is the limit of tensors of rank $r$,  but not of smaller rank, is said to have \emph{border rank} $r$. Thus the border rank of a tensor is always less than
 or equal to its rank.

For instance, while the tensors of complex rank 2 are dense among the $2\times 2 \times 2$ tensors, there is a unique $G(\C)$-orbit of rank 3 tensors \cite{dSL}, which therefore have border rank 2.  An orbit representative, called the Werner tensor in the physics literature, is usually taken as
$$\mathbf e_2 \otimes \mathbf e_1 \otimes \mathbf e_1 + \mathbf e_1 \otimes \mathbf e_2 \otimes \mathbf e_1 +\mathbf e_1 \otimes \mathbf e_1 \otimes \mathbf e_2.$$
For our purposes, it is more convenient to apply a permutation in the second index, so that its 3-slices become
$$W=\left [ \begin{pmatrix} 1&0\\0&1\end{pmatrix},\ \ \begin{pmatrix} 0&1\\0&0\end{pmatrix}\right ],$$
and thus have the form described in Proposition \ref{prop:orbitrep}.
One may express $W$ as an explicit limit of rank 2 tensors using a difference quotient \cite{dSL}. This difference quotient construction generalizes to other formats, to produce simple examples of tensors whose rank is larger than their border rank.

\smallskip

Proposition  \ref{prop:orbitrep} suggests a different way of obtaining $W$, and many other tensors whose rank exceeds their border rank. Our goal in this section is to provide some explicit examples.
\medskip

In the $n=3$ case, consider the tensor given by 3-slices as
$$K_3=\left [ \begin{pmatrix} 1&0&0\\0&1&0\\0&0&1\end{pmatrix},\ \ \begin{pmatrix} 0&1&0\\0&0&1\\0&0&0\end{pmatrix},\ \ \begin{pmatrix} 0&0&1\\0&0&0\\0&0&0\end{pmatrix}\right ],$$ and its perturbation
$$K_{3,\epsilon}=\left [ \begin{pmatrix} 1&0&0\\0&1&0\\0&0&1\end{pmatrix},\ \ \begin{pmatrix} 0&1&0\\0&\epsilon&1\\0&0&2\epsilon\end{pmatrix},\ \ \begin{pmatrix} 0&\epsilon&1\\0&\epsilon^2&3\epsilon\\0&0&4\epsilon^2\end{pmatrix}\right ].$$ 
Both tensors are 3-slice non-singular and have commuting slices, since for each the third slice is the square of the second. 

For arbitrary $n$, one can similarly construct $K_n$ with slices whose entries are all zeros except for successive super-diagonals of $1$s. Perturbing the diagonal of the second slice by adding $(0, \epsilon, 2\epsilon, 3\epsilon, \dots, (n-1)\epsilon)$, the other slices can be perturbed to be appropriate powers of the perturbed second slice, so that one obtains $K_{n,\epsilon}$ with all 3-slices commuting. The matrix $Z_{n,\epsilon}$ of diagonals of the slices of $K_{n,\epsilon}$ is then a Vandermonde matrix, and hence non-singular for $\epsilon\ne 0$.

Now $K_{n,\epsilon}$ meets the hypotheses of Proposition \ref{prop:orbitrep}, and has slices already upper-triangularized. Moreover since $Z_{n,\epsilon}$ is non-singular for $\epsilon\ne 0$, it follows that $K_{n,\epsilon} \in \mathcal D(\C)$.

Since $K_n$ is in the closure of all $K_{n,\epsilon}$, we see $K\in V_n$ has border rank at most $n$. Since the multlinear rank of $K_n$ is $(n,n,n)$, it cannot have border rank less than $n$, so its border rank is exactly $n$.
Since $f_3(K_n; \mathbf x)=0$, Theorem \ref{thm:fzero} shows $K_n\notin \mathcal D(\C)$. Since this orbit is precisely the tensors of rank $n$ and multilinear rank $(n,n,n)$, this
implies the  tensor rank of $K_n$ must be strictly greater than $n$.

We now determine the rank precisely.

\begin{proposition} \label{prop:rankjump} For any $n>0$, $K_n$ has border rank $n$ and rank $2n-1$, over $\mathbb C$.
\end{proposition}
\begin{proof}
The fact that $K_n$ has border rank $n$ has been discussed.

\smallskip

To show the rank is at most $2n-1$ we give an explicit representation, suggested by Anders Jensen. 
We work with a more symmetric tensor $K_n'$, obtained by acting on $K_n$ by a permutation in the second index, reversing the order of the columns of each slice. For example,
$$K_3'=\left [ \begin{pmatrix} 0&0&1\\0&1&0\\1&0&0\end{pmatrix},\ \ \begin{pmatrix} 0&1&0\\1&0&0\\0&0&0\end{pmatrix},\ \ \begin{pmatrix} 1&0&0\\0&0&0\\0&0&0\end{pmatrix}\right ].$$
In general $K_n'$ will be the $n\times n \times n$ tensor of all zeros, except for 1's in the $(i,j,k)$ position when $i+j+k=n+2$.

Let $\zeta$ denote a primitive $(2n-1)$th root of unity. Let $\mathbf v_l=(\zeta^l,\zeta^{2l},\dots, \zeta^{nl})$. Then we claim that
$$K_n'=\sum_{l=1}^{2n-1}\frac 1{2n-1} \zeta^{l(n-3)} \mathbf v_l \otimes\mathbf v_l \otimes\mathbf v_l,$$ and thus $K_n'$ has rank at most $2n-1$.
Indeed, the $(i,j,k)$ entry of this sum is
$$ \frac 1{2n-1} \sum_{l=1}^{2n-1} \zeta^{l(i+j+k+n-3)} =\begin{cases} 1& \text{ if  $(2n-1) \mid (i+j+k+n-3)$ }\\0&\text{otherwise}\end{cases}.$$ 
Since $n\le(i+j+k+n-3)\le 4n-3$, the non-zero entries occur only when $i+j+k+n-3=2n-1$, \emph{i.e.}, when $i+j+k=n+2$. 

\smallskip

To see the rank is at least $2n-1$, suppose $K_n$ could be expressed as 
$$ K_n=\sum_{l=1}^{2n-2} \mathbf u_l \otimes\mathbf v_l \otimes\mathbf w_l,$$
with $\mathbf u_i, \mathbf v_i,\mathbf w_i\in \mathbb C^n$. Since the $12|3$ flattening of $K_n$ has rank $n$, the $\mathbf w_i$ must span $\C^n$. Thus without loss of generality we may assume $\mathcal B=\{\mathbf w_1,\mathbf w_2,\dots,\mathbf w_n\}$ is a basis for $\C^n$.
Let $\mathcal B^*=\{\mathbf w^*_1,\mathbf w^*_2,\dots,\mathbf w^*_n\}$ be the dual basis. Then for any $i\in\{1,2,\dots,n\}$, 
$\mathbf w^*_i$ annihilates at least $n-1$ of the $\mathbf w_j$, so $K_n*_3\mathbf w^*_i$ is a matrix of rank at most $(2n-2)-(n-1)=n-1$.
But one sees from the explicit form of $K_n$ that $K_n*_3 \mathbf w_i^*$ can have rank at most $n-1$ only if $\mathbf w_i^*$ has first coordinate $0$. 
This contradicts that $\mathcal B^*$ is a basis.
\end{proof}

\medskip

One can construct many other examples of `rank jumping' by considering variants of the arguments above using different Jordan block structures of the slices.
For example,
$$L=\left [ \begin{pmatrix} 1&0&0\\0&1&0\\0&0&1\end{pmatrix},\ \ \begin{pmatrix} 0&1&0\\0&0&0\\0&0&0\end{pmatrix},\ \ \begin{pmatrix} 0&0&0\\0&0&0\\0&0&1\end{pmatrix}\right ]$$ can be
perturbed to
$$L_\epsilon=\left [ \begin{pmatrix} 1&0&0\\0&1&0\\0&0&1\end{pmatrix},\ \ \begin{pmatrix} 0&1&0\\0&\epsilon&0\\0&0&0\end{pmatrix},\ \ \begin{pmatrix} 0&0&0\\0&0&0\\0&0&1\end{pmatrix}\right ]\in \mathcal D(\C).$$ 
Here one can see that $L$ has tensor rank at most 4 (by subtracting the third slice from the first). Since it also has multilinear rank
$(3,3,3)$, and by Theorem \ref{thm:fzero} is not in $\mathcal D(\C)$, its tensor rank must be exactly 4.

\medskip

Finally, we note that
the maximal $\C$-rank of a $n\times n\times n$ tensor of border rank $n$ for $n=2$ is well known to be $3$. For $n=3$, it is claimed in \cite{Kruskal2} that the maximal rank is $5$. The tensor $K_n$ achieves these bounds in both cases. We know of no examples of $n\times n\times n$ tensors of border rank $n$ whose rank exceeds $2n-1$, the rank of $K_n$.  It has been conjectured by one of us (JAR, see \cite{BB}) that no such tensors exist.

It should be noted that the tensors $K_n$ described in this section are similar to those given in Theorem 5.6 of \cite{Alexeev} of size $n\times n\times (\lfloor \log_2 n \rfloor+1)$, whose rank is $2n-1$ when $n=2^k$, and whose border rank has been shown to be $n$ \cite{Land2}. 
(When $n\ne 2^k$, the rank of those tensors is slightly smaller than $2n-1$.)
Related examples are also given in Corollary 5.7 of \cite{Alexeev} of  tensors of size $n\times (n+1)\times(n+1)$ and rank
approximately $3n$. However the border rank of these has only been shown to be bounded above by approximately $2n$ when $n=2^k$ \cite{Land2}, with the precise border rank unknown. Thus it is unclear what the gap between rank and border rank is for these last examples.

\section{Acknowledgements}

Much of the research in this work was conducted while ESA and JAR were in residence at the University of Tasmania as Visiting Scholars; they thank the University and the School of Mathematics and Physics for financial support and hospitality.
ESA and JAR also thank the Institut Mittag Leffler and the organizers and participants in its Spring 2011 Program in Algebraic Geometry with a View toward Applications for additional support and stimulating conversations. We thank Allesandra Bernardi especially for many helpful clarifications in discussions, and Anders Jensen for his contribution to the argument for Proposition \ref{prop:rankjump}.

\bibliographystyle{siam}

\bibliography{tensor}

\end{document}